\documentclass{amsart}

\usepackage{amssymb}
\usepackage{amsbsy}
\usepackage{amscd}
\usepackage{amsmath}
\usepackage{amsthm}
\usepackage{amsxtra}
\usepackage{latexsym}

\usepackage{mathrsfs}
\usepackage{upgreek}
\usepackage{wasysym}
\usepackage[all,cmtip]{xy}
\usepackage{xcolor}
\definecolor{rouge}{rgb}{0.85,0.1,.4}
\definecolor{bleu}{rgb}{0.1,0.2,0.9}
\definecolor{violet}{rgb}{0.7,0,0.8}

\usepackage[dvipdfmx,
colorlinks=true,linkcolor=bleu,urlcolor=violet,citecolor=rouge]{hyperref}

\newcommand{\ket}{{\rangle}}

\newcommand{\cprime}{$'$}
\newcommand{\on}{\operatorname}

\newcommand{\mc}{\mathcal}
\newcommand{\mf}{\mathfrak}

\newcommand{\affg}{\widehat{\mf{g}}}

\newcommand{\isomap}{\stackrel{\simeq\,}{\longrightarrow}}
\newcommand{\Z}{\mathbb{Z}}
\newcommand{\C}{\mathbb{C}}

\newcommand{\W}{\mathcal{W}}  
\newcommand{\ra}{\rightarrow}
\newcommand{\lam}{\lambda}

\newcommand{\vac}{{|0\rangle}}
\newcommand{\haru}{\operatorname{span}}

\newcommand{\LL}{\mc{L}}

\def\g{\mathfrak{g}}

\def\z{\mathfrak{z}}
\def\h{\mathfrak{h}}
\def\n{\mathfrak{n}}
\def\m{\mathfrak{m}}
\def\p{\mathfrak{p}}

\def\Slo{\mathscr{S}}

\def\Zent{{\mathcal{Z}}}

\def\J{J_\infty}

\def\P{\mathscr{P}}
\def\sl{\mathfrak{sl}}

\def\d{d}

\def\W{\mathcal{W}}

\def\P{\mathscr{P}}
\def\V{\mathscr{V}}
\def\Core{\mathscr{C}}

\def\le{\leqslant}
\def\ge{\geqslant}

\DeclareMathOperator{\End}{End}
\DeclareMathOperator{\Spec}{Spec}
\DeclareMathOperator{\Specm}{Specm}
\DeclareMathOperator{\gr}{gr}

\DeclareMathOperator{\Hom}{Hom}

\theoremstyle{theorem}
\newtheorem{Th}{Theorem}[section]

\newtheorem{Pro}[Th]{Proposition}

\newtheorem{Lem}[Th]{Lemma}

\newtheorem{Co}[Th]{Corollary}

\theoremstyle{remark}
\newtheorem{Def}[Th]{Definition}

\newtheorem{Rem}[Th]{Remark}

\newlength{\larg}
\title{Arc spaces and chiral symplectic cores}

\date{\today}
\subjclass[2010]{}
\author[Tomoyuki Arakawa]{Tomoyuki Arakawa\textsuperscript{1}}
\address{\textsuperscript{1}Research Institute for Mathematical Sciences
\\ Kyoto University\\ 
 Kyoto 606-8502 JAPAN}
 
\email{arakawa@kurims.kyoto-u.ac.jp}
\author[Anne Moreau]{Anne Moreau\textsuperscript{2}}
\address{\textsuperscript{2}Laboratoire Paul Painlev\'{e}\\ 
Universit\'{e} de Lille 
\\ 59655 Villeneuve d'Ascq Cedex\\ France}
\email{anne.moreau@univ-lille.fr}

\begin{document}

\maketitle

\begin{center}
{\em Dedicated to the 70th birthday of Professor Masaki Kashiwara}
\end{center}

\bigskip

\begin{abstract}
We introduce the notion of chiral symplectic cores in a vertex Poisson 
variety,   
which can be viewed as analogs of symplectic leaves in Poisson varieties. 
As an application we show that any quasi-lisse vertex algebra is a quantization of the arc space of its associated variety, in the sense that its reduced singular support coincides with the 
reduced arc space of its associated variety.
We also show that the coordinate ring of the arc space of Slodowy slices is free over
its vertex Poisson center, and the latter
coincides with the vertex Poisson center
of the coordinate ring of the arc space of the dual of the corresponding simple Lie algebra. 
\end{abstract}

\section{Introduction}
Any vertex algebra is canonically filtered \cite{Li05}, 
and hence can be viewed as a quantization of its associated graded vertex  Poisson  algebra.
Since the structure of a vertex algebra is usually quite complicated, it is often very useful to reduce a problem of a vertex algebra to that of the 
geometry of
the associated vertex Poisson scheme,
that is,
the spectrum of the 
 associated  graded vertex  Poisson  algebra (see e.g. 
 \cite{Fre07,Ara09b,A2012Dec}).
Since a vertex Poisson scheme  can be regarded as a chiral analogue of a Poisson scheme,
it is natural to try to upgrade notions in Poisson geometry to the setting of vertex Poisson schemes.
We note that the {\em arc space} $J_{\infty}X$ of an affine Poisson scheme $X$ is a basic example of 
vertex Poisson  schemes (\cite{Ara12}).

In \cite{Brown-Gordon} Brown and Gordon introduced the notion of {\em symplectic cores}
in a Poisson variety 
{which is expected to be} 
the finest possible {\em algebraic} stratification in which the Hamiltonian vector fields
are tangent,
and showed that the symplectic cores
in fact
coincide with
the symplectic leaves if there is only finitely many numbers of symplectic leaves.
In this paper we introduce the notion of {\em chiral symplectic  cores}
in a vertex Poisson scheme,
which we expect to be the 
finest possible algebraic stratification in which the {\em chiral} Hamiltonian vector fields
are tangent.

We have two major applications of the notion of chiral symplectic   cores.

First, recall that 
a vertex algebra $V$ is called {\em quasi-lisse}
if its associated variety $X_V$ has 
finitely many symplectic leaves (\cite{Arakawam:kq}).
For instance, a simple affine vertex algebra $V$ associated with 
a simple Lie algebra $\g$ is quasi-lisse if and only if $X_V$ is contained in the nilpotent cone of $\g$.
Therefore \cite{Ara09b},
simple admissible affine vertex algebras are quasi-lisse. 
We refer to \cite{AM15,AraMor17,AraMor16b} for 
other examples of simple quasi-lisse vertex algebras. 
Furthermore,
all the   vertex algebras obtained from {\em four-dimensional}  $\mathcal{N}=2$ superconformal field theories 
(\cite{BeeLemLie15}) 
are expected to be quasi-lisse (\cite{A.Higgs,BeeRas})
see e.g.\ \cite{BPRvR,LP,SXY,BKN,Cre,BMR,A6} for examples of vertex algebras obtained from 4d $\mathcal{N}=2$ SCFTs.
It is also believed in physics that there exist
Higgs branch vertex algebras and Column branch vertex algebras
in {\em three-dimensional} gauge theories that 
are expected to be quasi-lisse as well (\cite{CCG}).

We show that any quasi-lisse vertex algebra $V$ is a quantization of the {reduced arc space 
of its associated variety, in the sense that its reduced singular support 
$\Specm (\gr V)$ coincides with 
$\J  X_V$ as topological spaces (Theorem~\ref{Th:quasi-lisse})}. 
Moreover, for a quasi-lisse vertex algebra $V$, we show that 
each irreducible component of 
$\J  X_V$ (there are finitely many of them) is a symplectic core closure 
(Theorem~\ref{Th:quasi-lisse}). 

Second,
let $\g$ be a complex simple Lie algebra with adjoint group $G$. 
We identify $\g$ with its dual $\g^*$ through the Killing 
form of $\g$. 
Denote by $\Slo_f$ the {\em Slodowy slice} 
$f+\g^{e}$ associated with an $\sl_2$-triple $(e,h,f)$ of $\g$. 
The affine variety 
 $\Slo_f$ has a Poisson structure obtained from that of $\g^*$ 
by Hamiltonian reduction \cite{GanGin02}. 
Consider 
the adjoint quotient morphism
$$\psi_f \colon \Slo_f \to \g^*/ \!/ G.$$
It is known \cite{Pre02} that
any  fiber $\psi_f^{-1}(\xi)$ of this morphism  is the closure of a symplectic leave,
which 
is irreducible and reduced.
We show that 
any fiber of the 
 induced vertex Poisson algebra morphism
$$\J\psi_f \colon \J\Slo_f \to  \J(\g^*/ \!/ G)
$$ 
is an irreducible and reduced chiral Poisson subscheme of $\J \Slo_f$. 
This result enables us to show that
the morphism $(\J \psi_f )^*$ induces an {\em isomorphism} 
of vertex Poisson algebras 
between $\C[\J \g^*]^{\J G}$ and the vertex  Poisson center of 
$\C[\J \Slo_f]$,
and
 that $\C[\J \Slo_f]$
 is free over its vertex Poisson center (Theorem \ref{Th:main-result}). 
  As a consequence, we obtain that 
the center of the {\em affine W-algebra} $\W^{cri}(\g,f)$ 
associated with $(\g,f)$ at the critical level is identified with the 
{\em Feigin-Frenkel center} 
$\z(\affg)$, that is, the center 
of the affine vertex algebra $V^{cri}(\g)$ at the critical level 
(cf.~Theorem \ref{Th:W-algebra}). 
This later fact was claimed in \cite{A11} but the proof was incomplete. 
We take the opportunity of this work to clarify this point. 

\subsection*{Acknowledgements} The authors are very grateful to Kenneth Brown 
for explaining his paper with Iain Gordon. The second author also thanks 
Jean-Yves Charbonnel for useful comments. 

The first named author is partially supported by  JSPS KAKENHI Grant Numbers 17H01086, 17K18724.
The second named author is supported in part by the ANR Project GeoLie Grant number ANR-15-CE40-0012, and in part by the Labex CEMPI  (ANR-11-LABX-0007-01). 

\subsection*{Notations}
The topology is always the Zariski topology. 
So the term {\em closure} always refers to the Zariski closure. 

\section{Vertex algebras} \label{sec:vertex-algebras}
Let $V$ be a vector space over $\C$. 
%

\begin{Def}
The vector space $V$ is called a {\em vertex algebra} if 
it is equipped with 
 the following data: 
 \begin{itemize}
 \item ({\em the vacuum vector}) a vector $|0\ket \in V$, 
 \item ({\em the vertex operators}) a linear map 
 $$V \to ({\rm End}\,V)[[z,z^{-1}]], \quad a \mapsto 
 a(z)=\sum_{n\in\Z} a_{(n)}z^{-n-1},$$
 such that for all $a,b \in V$, 
 $a_{(n)}b=0$ for $n$ sufficiently large. 
 \item ({\em the translation operator}) a linear map 
 $T \colon V \to V$.  
 \end{itemize}
These data are subject to the following axioms: 
\begin{itemize}
\item $|0\ket (z)={\rm id}_V$. 
Furthermore, for all $a \in V$, $a(z)|0\ket \in V[[z]]$ and 
\hbox{$\lim\limits_{z\ra 0}a(z)|0\ket =a$.} 
\item for any $a \in V$, 
$$[T,a(z)]=\partial_z a( z),$$
and $T|0\ket=0$.
\item for all $a,b\in V$,
$(z-w)^{N_{a,b}}[a(z),b(w)]=0$ for some $N_{a,b}\in \Z_{\ge 0}$. 
 \end{itemize}
 \end{Def}
Assume from now that $V$ is a vertex algebra. 
A consequence of the definition are
the following relations, called {\em Borcherds identities}:
\begin{align}
 &[a_{(m)}, b_{(n)}]
 =\sum_{i\ge 0}\begin{pmatrix}
				 m\\i
		\end{pmatrix}(a_{(i)}b)_{(m+n-i)}, 
\label{eq:com-formula}\\ \label{eq:com-formula2}
&(a_{(m)}b)_{(n)}=\sum_{j\ge 0}(-1)^j
\begin{pmatrix}
m\\j
\end{pmatrix} (a_{(m-j)}b_{(n+j)}-(-1)^mb_{(m+n-j)}a_{(j)}), 
\end{align}
for $m,n\in\Z$. 

A {\em vertex ideal} $I$ of $V$ is a $T$-invariant subspace of $V$ 
such that $a_{(n)}b \in I$ for all $a \in I$, $b \in V$. 
By the {\em skew-symmetry 
property} which says that for all $a,b \in V$, the identity 
$$a(z)b=e^{zT}b(-z)a$$ holds in $V((z)),$ 
 a vertex ideal $I$ of $V$ is also a $T$-invariant subspace of $V$ 
such that $b_{(n)}a \in I$ for all $a \in I$, $b \in V$. 


The vertex algebra $V$ is called {\em commutative}
if all vertex operators $a(z)$, $a \in V$, commute each other, 
that is,
%
\begin{align*}
 [a_{(m)}, b_{(n)}]=0,\qquad\forall a,b\in \Z,\ m,n\in \Z.
\end{align*}
By 
\eqref{eq:com-formula},
$V$ is a commutative vertex algebra if and only if
$a(z) \in \End V[[z]]$ for all $a \in V$. 

A commutative vertex algebra has a structure of a unital commutative 
algebra with the product: 
\begin{align*}
 a \cdot b=a_{(-1)}b,
\end{align*}
where the unit is given by the vacuum vector $|0\ket$.
The translation operator $T$ of $V$
acts on $V$ as a derivation with respect to this product:
\begin{align*}
 T(a\cdot  b)=(Ta)\cdot b+a\cdot (Tb).
\end{align*}
Therefore 
a commutative vertex algebra 
has the structure of a
 {\em differential algebra},
that is,
 a unital commutative algebra
equipped with a derivation.

Conversely,
there is a unique vertex algebra structure on
a differential algebra $R$ with derivation $\partial$ 
such that: 
\begin{align*}
a(z)b= \left(e^{z \partial}a\right)b =\sum_{n\ge 0} 
\displaystyle{\frac{z^n}{n!}} (\partial^n a) b,
\end{align*}
for all $a,b\in R$. We take the unit as the vacuum vector. 
This correspondence gives that the 
category of commutative vertex algebras
is the same as that of differential algebras \cite{Bor86}.


\section{Jet schemes and arc spaces} \label{sec:arc}
Our main references about jet schemes 
and arc spaces are \cite{Mus01,EinMus,Ishii}.

Denote by $Sch$ the category of schemes of finite 
type over $\C$. 
Let $X$ be an object of this category, 
and $n \in\Z_{\ge 0}$. 

\begin{Def} \label{d:jet} 
An {\em $n$-jet of $X$} is a morphism
$$\Spec \C[t]/(t^{n+1}) \longrightarrow X.$$
The set of all $n$-jets of $X$ carries the structure of a 
scheme $J_n X$, 
called the {\em $n$-th jet scheme of $X$}. 
It is a scheme of finite type over $\C$ characterized by the following functorial property: 
for every scheme $Z$ over $\C$, we have  
$$\Hom_{Sch}(Z,J_n X) = \Hom_{Sch}(Z\times_{\Spec\C} 
\Spec \C[t]/(t^{n+1}),X). 
$$
\end{Def}
The $\C$-points of $J_n X$ are thus 
the $\C[t]/(t^{n+1})$-points of $X$. 
From Definition~\ref{d:jet}, we have for example that $ J_0 X \simeq X$ 
and that $ J_1 X\simeq {\rm T}X$, where ${\rm T}X$ 
denotes the total tangent bundle of $X$. 

The canonical projection $\C[t]/(t^{m+1})\to \C[t]/(t^{n+1})$, $m \ge n$,  
induces a {\em truncation morphism} 
$$\pi^{X}_{m,n}\colon J_m X \rightarrow  J_n X.$$ 

Define the  (formal) disc
as
\begin{align*}
 D:=\Spec \C[[t]].
\end{align*} 
The projections $\pi ^{X}_{m,n}$ yield a projective system 
$\{J_m X, \pi ^{X}_{m,n}\}_{m\ge n}$ of schemes. 
\begin{Def}
Denote by $ J_{\infty} X$ its projective limit 
in the category of schemes, 
$$ J_{\infty} X= \varprojlim J_n X.$$
It is called the {\em arc space}, or {\em the infinite jet scheme},  
of $X$. 
\end{Def}
Thus elements of $ \J X$ are the morphisms  
$$\gamma \colon D \to X,$$
and for every scheme $Z$ over $\C$, 
$$\Hom_{Sch}(Z,J_\infty X) = \Hom_{Sch}(Z\widehat{\times}_{\Spec\C} 
D,X),$$
where $Z\widehat{\times}_{\Spec\C} 
D$ means the formal completion of $Z{\times}_{\Spec\C} 
D$ along the subscheme 
$Z{\times}_{\Spec\C} \{0\}$. 
In other words, the contravariant functor 
$$Sch \to Set, \quad Z\mapsto \Hom_{Sch}(Z\widehat{\times}_{\Spec\C}  D,X)$$ 
is represented by the scheme $J_\infty X$. 

We denote by $\pi ^{X}_{\infty,n}$ 
the morphism: 
$$\pi ^{X}_{\infty,n} \colon  J_\infty X \to J_n X.$$ 
It is surjective if $X$ is smooth. 
The canonical injection $\C \hookrightarrow \C[[t]]$  
induces a morphism $\iota ^{X}_{\infty} \colon X \to \J X$, and we have 
$ \pi ^{X}_{\infty,0} \circ  \iota ^{X}_{\infty}={\rm id}_X$. Hence $ \iota ^{X}_{\infty}$ 
is injective and  $\pi ^{X}_{\infty,0}$ is surjective (for any $X$). 
Similarly, the canonical injection $\C \hookrightarrow \C[t]/{(t^{n+1})}$  
induces a morphism $\iota ^{X}_{n} \colon X \to J_n X$, and we have 
$ \pi ^{X}_{n,0} \circ  \iota ^{X}_{n}={\rm id}_X$. Hence $ \iota ^{X}_{n}$ 
is injective and  $\pi ^{X}_{n,0}$ is surjective (for any $X$). 

When the variety $X$ is obvious, we simply write $\pi_{m,n}, 
\pi_{\infty,n}$, $ \iota_{n},\iota _{\infty}, \ldots$ 
for $\pi_{m,n}^X, \pi_{\infty,n}^X$, $\iota ^{X}_{n}, \iota ^{X}_{\infty},\ldots$. 

 
In the case where $X =\Spec \C[x^1,\ldots,x^N] \cong \mathbb{A}^N$, 
$N\in \Z_{>0}$, is an affine space, 
we have the following explicit description of $ J_{\infty} X$. 
Giving a morphism
$\gamma \colon D \to \mathbb{A}^N$ is equivalent to giving a morphism 
$\gamma^*\colon \C[x^1,\ldots,x^N] \to\C[[t]]$, 
or  to giving 
$$\gamma^*(x^{i}) =\sum\limits_{j \ge 0} \gamma^{i}_{(-j-1)} t^j ,
\qquad i=1,\ldots,N.$$
Define functions over $J_{\infty} \mathbb{A}^N$ by setting for $i=1,\ldots,N$:  
$$x^{i}_{(-j-1)} ({\gamma}) = j! \gamma^{i}_{(-j-1)}.$$
Then 
$$ J_{\infty} \mathbb{A}^N = \Spec \C [ x^{i}_{(-j-1)}  \, ;\,  i=1,\ldots,N,\, j \ge 0].$$ 
Define a derivation $T$ of the algebra $\C[ x^{i}_{(-j-1)} \; ;\; 
i=1,\ldots,N, \, j\ge 0]$ by 
$$T x^{i}_{(-j)} =  j x^{i}_{(-j-1)},\qquad j > 0.
$$ 
Here we identify $x^{i}$ with $x^{i}_{(-1)}$. 

More generally, 
if $X \subset \mathbb{A}^N$ is an affine 
subscheme defined by an ideal $I =( 
f_1,\ldots,f_r)$ of $\C[x^1,\ldots,x^N]$, that is, 
$X=\Spec R$ 
with 
$$R= {\C[x^1, x^2, \cdots, x^N]}/{\left( f_1, f_2, \cdots, f_r\right)},$$ 
then its arc space
 $ J_\infty X$ is
the affine scheme
$\Spec ( J_\infty R)$, 
where
 \begin{align}
 \label{eq:arc-ring} J_\infty R:= \frac{\C[x^{i}_{(-j-1)} \; ;\;  
 i=1, 2, \cdots, N,\, j\ge 0]}{\left(  T^j f_i \; ;\;  i=1, \ldots, r,\, j\ge 0
\right)}, 
 \end{align}
and $T$ is as defined above. 

Similarly, we have for any  $n \in\Z_{\ge 0}$,
 \begin{align}
 \label{eq:arc-jets} 
 J_n R:= \frac{\C[x^{i}_{(-j-1)} \; ;\;  
 i=1, 2, \cdots, N,\, j=0,\ldots,n]}{\left( T^j f_i \; ;\;  i=1, \ldots, r,\, j=0,
 \ldots,n \right)}.
  \end{align}

The derivation $T$ acts on the quotient ring $ J_\infty R$ 
given by \eqref{eq:arc-ring}.
Hence
for an affine scheme $X=\Spec R$,
the coordinate ring  $J_\infty R=\C[J_{\infty} X]$
of its arc space $ J_\infty X$ 
is a differential algebra, hence is a commutative vertex algebra.  

\begin{Rem}[{\cite{EinMus}}]\label{Rem:universal property of JR}
The differential algebra $( J_\infty R,T)$ 
is universal in the following sense. 
We have a $\C$-algebra homomorphism 
$j \colon R \to  J_\infty R$ such that 
if $(A,\partial)$ is another differential algebra, 
and if $f \colon R \to A$ is a $\C$-algebra 
homomorphism, then there is a unique 
differential algebra homomorphism 
$h \colon  J_\infty R \to A$ making the following 
diagram commutative: 
\begin{align*}
\xymatrix{R \ar[rr]^{j} \ar[rd]_{f} && ( J_\infty R,T) 
\ar@{-->}[ld]^{h}\\ 
&(A,\partial)&}
\end{align*}
\end{Rem}

The map from a scheme to its $n$-th jet schemes and arc space is functorial. 
If $f \colon X \to Y$ is a morphism of schemes, then we naturally obtain 
a morphism 
$J_n f \colon J_n X \to J_n Y$ making the following diagram commutative,

\smallskip
\begin{center}
~\xymatrix{J_n X \ar[r]^{J_n f }\ar[d]_{\pi ^{X}_{n,0}} & J_n Y \ar[d]^{\pi_{n,0}^Y}\\
X \ar[r]_{f} & Y }
\end{center}

We also have the following 
for every schemes $X,Y$, 
\begin{align} \label{eq:product-jets}
J_n(X\times Y)\cong J_n X\times J_n Y. 
\end{align}
If $A$ is a group scheme over $\C$, 
then $J_n A$ is also a group 
scheme over $\C$. Moreover, 
by (\ref{eq:product-jets}), if $A$ acts on $X$,  
then $J_n A$ acts on $J_n X$. 

From now on, whenever dealing with the schemes $J_n X$ and $\J X$ 
we will restrict to their $\C$-valued points, unless otherwise specified. 
Since the ground field is $\C$, $\C$-valued points corresponds to maximal 
ideals \cite[Proposition 2.10]{Ishii_toric}. 

Denote by $X_{\on{red}}$ the reduced scheme of $X$.
Since $\C[[t]]$ is a domain, we have
$\Hom(\Spec \C[[t]],X) = \Hom(\Spec \C[[t]],X_{\on{red}})$. 
Hence, the natural morphism $X_{\on{red}}\ra X$ induces 
an isomorphism $ \J X_{\on{red}}\isomap \J X$
of topological spaces.
(Note that the analogous assertion is false for the spaces $J_n X$.) 
Similarly, if $X = X_1 \cup \ldots \cup X_r$, where
all $X_i$ are closed in $X$, then
$$\J X = \J X_1 \cup \ldots \cup \J X_r.$$





Moreover, we have the following result 
(which is false for the jet spaces $J_n X$).

\begin{Th}[Kolchin {\cite{Kol73}}] \label{Th:Kolchin}
The arc space $ \J X$ is irreducible 
if $X$ is irreducible. 
\end{Th}

More precisely, we have for any $n \in\Z_{\ge 0}$, 
\begin{align} \label{eq:arc-irred}
J_n X = \pi_{n,0}^{-1}(X_{sing}) \cup \overline{\pi_{n,0}^{-1}(X_{reg})},   
\end{align}
and $\overline{\pi_{n,0}^{-1}(X_{reg})}$ is an irreducible component 
of $J_n X$. Here, $X_{sing}$ denotes the singular locus of $X$, 
and $X_{reg}$ its open complement in $X$. 
Kolchin's theorem says that 
$$\J X = \overline{\pi_{\infty,0}^{-1}(X_{reg})}.$$


Let $n \in \Z_{\ge 0} \cup \{\infty\}$. 
The natural projection $\pi ^{X}_{n,0}\colon  J_n X\ra X$ 
 corresponds to the embedding $R \hookrightarrow  J_n R$, 
 $x^{i}\ra x_{(-1)}^{i}$ 
 in the case where $X=\Spec R$ is affine.
{If $\m$ is a maximal ideal of $J_n R$, note that 
 $\pi_{n,0}^X(\m)=\m \cap R$.} 

For $I$ an ideal of $R$, we denote 
by $J_n I$ the smallest $T$-stable 
ideal of $J_n R$ containing $I$, that is, 
$J_n I$ is generated by the elements 
$T^j a$, $j=0,\ldots,n$, $a \in I$. 
{Recall that $\iota_n^X$ 
denotes the embedding $X \hookrightarrow J_n X$, 
and observe that $\iota_n^X (\m)= J_n (\m)$, for $\m$ a maximal ideal 
of $R$.}

\section{Vertex Poisson algebras 
and chiral Poisson ideals} \label{sec:PVA}
%
\begin{Def} \label{def:PVA}
A commutative vertex algebra $V$ 
is called a {\em vertex Poisson algebra} 
if it is also equipped with a linear operation, 
$$V \to \Hom (V,z^{-1}V[z^{-1}]), \quad 
a \mapsto a_-(z),$$
such that 
\begin{align} \label{eq:PVA-1}
& (T a)_{(n)} = -n a_{(n-1)}, & \\\label{eq:PVA-2}
& a_{(n)}b = \sum_{j \ge 0} 
(-1)^{n+j+1} \dfrac{1}{j!} T^{j} (b_{(n+j)} a) , & \\ \label{eq:PVA-3}
& [a_{(m)},b_{(n)}] =\sum_{j\ge 0} 
\begin{pmatrix} m \\ 
j \end{pmatrix} (a_{(j)} b)_{(m+n- j)}, & \\ \label{eq:PVA-4}
&a_{(n)}(b \cdot c) = (a_{(n)} b) \cdot c + b \cdot (a_{(n)}c)& 
\end{align}
for $a,b,c \in V$ and $n,m \ge 0$. 
Here, by abuse of notations, we have set 
$$a_-(z)=\sum_{n\ge 0} a_{(n)} z^{-n-1}$$ 
so that the $a_{(n)}$, $n \ge 0$, are ^^ ^^ new" operators, 
the ^^ ^^ old" ones given by the field $a(z)$ being
zero for $n \ge 0$ since $V$ is commutative. 
\end{Def}

%

The equation \eqref{eq:PVA-4}
says that $a_{(n)}$, $n\ge 0$, is a derivation of the ring $V$.
Note that 
\eqref{eq:PVA-2},
\eqref{eq:PVA-3} and
\eqref{eq:PVA-4}
are equivalent to the
^^ ^^ skewsymmetry",
the
^^ ^^ Jacobi identity"
and the ^^ ^^ left Leibniz rule"
in \cite[\S5.1]{Kaclecture17}.

It follows from the definition, that we also have  
the ^^ ^^ right Leibniz rule'' (\cite[Exercise 4.2]{Kaclecture17}): 
\begin{align} \label{eq:PVA-5}
(a \cdot b)_{(n)}c=\sum_{i \ge 0} (b_{(-i-1)}a_{(n+i)}c+a_{(-i-1)}b_{(n+i)}c), 
\end{align}
for all $a,b,c \in V$, $n \in \Z_{\ge 0}$. 



Arc spaces over an affine Poisson scheme 
naturally give rise to a vertex Poisson algebras, 
as shows the following result. 

\begin{Th}[{\cite[Proposition 2.3.1]{Ara12}}] \label{Th:Poisson-structure-arc-spaces}
Let $X$ be an affine Poisson scheme, that is,  $X=\Spec R$ for some Poisson
 algebra $R$. Then there is a unique vertex Poisson algebra structure on
 $ J_\infty R=\C[ \J X]$ such that 
$$a_{(n)}b=\begin{cases} \{a,b\} & \text{ if }n=0\\
0 & \text{ if }n >0,
\end{cases}$$
for $a,b \in R$.
\end{Th}

Let $V$ be a vertex Poisson algebra, and $I$ an ideal of 
$V$ in the associative sense. 
\begin{Def}
We say that $I$ is a {\em chiral Poisson ideal} of $V$ if 
$a_{(n)}I \subset I$ for all $a \in V$, $n\in\Z_{\ge 0}$. 
\end{Def}
A {\em vertex Poisson ideal} of $V$ is 
a chiral Poisson ideal that is stable under the action of $T$.
The quotient space  $V/I$ inherits a vertex Poisson algebra structure from 
$V$ if $I$ is a vertex Poisson ideal. 

\begin{Lem}[{\cite[3.3.2]{Dix96}}]
\label{Lem:Poisson-ideal}
If $I$ is a vertex (resp.~chiral) Poisson ideal of $V$, then so is 
its radical $\sqrt{I}$. 
\end{Lem}

 \begin{Def}Let $V$ be a vertex Poisson algebra. 
 We denote by $\Zent(V)$ the {\em vertex Poisson center} of $V$: 
 \begin{align*}
 \Zent(V) :=\{z \in V \mid z_{(n)}a=0, \, \forall\,a \in V,\, n \ge 0\} 
  \end{align*}
 By \eqref{eq:PVA-2}, we also have 
 $ \Zent(V) = \{z \in V \mid a_{(n)}z=0, \, \forall\,a \in V,\, n \ge 0\}$. 
 \end{Def}
 The  vertex Poisson center $\Zent(V)$ is a 
 vertex Poisson ideal of $V$. 
 Indeed, it is clearly invariant by the derivations $a_{(n)}$, $a\in V$,  
 $n\in\Z_{\ge 0}$. Moreover, it is invariant by $T$ by the axiom \eqref{eq:PVA-1}.

{We say that a scheme $X$ is a {\em vertex Poisson scheme} 
if its structure sheaf $\mc{O}_X$ is a sheaf of 
vertex Poisson algebras. 
If $X=\Spec V$ is an affine {vertex Poisson scheme}
and $I$ is a chiral Poisson ideal of $V$, we call 
the spectrum 
$\Spec (V/I)$
a {\em chiral Poisson subscheme} of $X$. 
A {\em chiral Poisson scheme} is a chiral Poisson subscheme 
of some vertex Poisson scheme. 

By Lemma \ref{Lem:Poisson-ideal},
the reduced scheme of a vertex Poisson scheme (resp.~chiral Poisson scheme) 
is also a vertex Poisson scheme (resp.~chiral Poisson scheme). 
In this case, we rather call it a {\em vertex Poisson variety} or a 
{\em chiral Poisson variety}.


\begin{Lem} \label{Lem:chiral Poisson-def} 
Let $I$ be an ideal of 
$J_{\infty} R$ in the associative sense. Then $I$ is a {chiral Poisson ideal} of $J_{\infty} R$ if 
and only if $a_{(n)}I \subset I$ for all $a \in R$, $n\in\Z_{\ge 0}$. 
\end{Lem}

\begin{proof}
The ^^ ^^ only if'' part is obvious. 

Assume that $a_{(n)} I  \subset I$ for all $a \in R$, $n\in\Z_{\ge 0}$.
 We wish to
show that $a_{(n)} I  \subset I$ for all $a \in J_\infty R$, $n\in\Z_{\ge 0}$.
Let $u \in I$. 
First, by \eqref{eq:PVA-1},
$$(T^j a)_{(n)} u = 
\begin{cases} 
(-1)^n \dfrac{n!}{(n-j) !} a_{(n-j)} u & \text{if } 0 \le j \le n,\\
0 & \text{if }j > n,
\end{cases}$$   
for all $a \in R$, $n,j \in\Z_{\ge 0}$. 
Hence $(T^j a)_{(n)} u \in I$ for all $a \in R$, $n,j \in\Z_{\ge 0}$ 
by our assumption. 
Next, by \eqref{eq:PVA-5} and the above,  $(a\cdot b)_{(n)} u$ 
is in $I$ for all $a,b$ of the form $T^j v$, $v \in R$. 
Since $\J R$ is generated as a commutative algebra 
by the elements $T^j v$, $v \in R$, we get the expected statement. 
\end{proof}

\section{Rank stratification} \label{sec:Rank}
Let $X=\Spec R$ be a reduced Poisson scheme, 
and $\{x^1,\dots, x^r\}$  a generating set for $R$. 
Let $n \in \Z_{\ge 0}$. 
Then $\{T^j x^{i}\mid i=1,\dots, r, j= 0,\ldots,n\}$ is a generating set for $J_n R=\C[J_n X]$.
We have by the equality (10) of \cite{Ara12} that 
for any $x,y \in R$, 
 \begin{align} \label{eq:10-Ara12}
 x_{(k)}(T^{l} y) 
 =\begin{cases} \frac{l!}{(l-k)!} T^{l-k}\{x,y\} & \text{if } l \ge k, \\
 0 & \text{otherwise}.
 \end{cases}
 \end{align}
Hence, for $x \in R$, the derivations $x_{(k)}$ of $J_{\infty} R$  
acts on $J_n R$ if $k\in \{0,\ldots,n\}$ 
by the description \eqref{eq:arc-jets} of $J_n R$.

Consider the 
$(n+1)r$-size square matrix
\begin{align*}
\mathscr{M}_n=\left(x^i_{(p)}(T^q x^{j})\right)_{1\le i,j\le r,\ 0\le p,q\le n}\in \on{Mat}_{(n+1)r}(J_n R).
\end{align*}
For $x \in J_nX$,
set $$\mathscr{M}_n(x)=\left(x^i_{(p)}(T^q x^{j})+\mf{m}_x\right)_{1\le i,j\le r,\ 0\le p,q\le n}\in \on{Mat}_{(n+1)r}(\C),$$
where $\m_x$ is the maximal ideal of {$J_n R$} corresponding to $x$. 
\begin{Lem}\label{Lem:crucial}
Let $x \in J_n X$. 
We have
$$\on{rank}\mathscr{M}_n(x)=(n+1)\on{rank}\mathscr{M}_0(\pi_{n,0}^X(x)).$$
In particular,
$\on{rank}\mathscr{M}_n(x)$ is independent of the choice of generators $\{x^1,\dots, x^r\}$
and depends only on $\pi_{n,0}^X(x)\in X$.
\end{Lem}
\begin{proof}
 By definition,  
$\mathscr{M}_n$  
is the following matrix: 
$$\begin{pmatrix}
x^{1}_{(0)}x^{1} & \cdots & x^{1}_{(0)}x^{r} & \cdots \cdots& x^{1}_{(0)}(T^{n}x^{1}) & \cdots 
& x^{1}_{(0)}(T^{n}x^{r}) \\
\vdots &&&&&& \vdots \\
x^{r}_{(0)}x^{1} & \cdots &  x^{r}_{(0)}x^{r} & \cdots\cdots & x^{r}_{(0)}(T^{n}x^{1}) & \cdots 
& x^{r}_{(0)}(T^{n}x^{r})\\
\vdots &&&&&& \vdots \\
\vdots &&&&&& \vdots \\
x^{1}_{(n)}x^{1} & \cdots & x^{1}_{(n)}x^{r} & \cdots \cdots& x^{1}_{(n)}(T^{n}x^{1}) & \cdots 
& x^{1}_{(n)}(T^{n}x^{r}) \\
\vdots &&&&&& \vdots \\
x^{r}_{(n)}x^{1} & \cdots &  x^{r}_{(n)}x^{r} & \cdots\cdots & x^{r}_{(n)}(T^{n}x^{1}) & \cdots 
& x^{r}_{(n)}(T^{n}x^{r})
\end{pmatrix}
$$
So by \eqref{eq:10-Ara12} it 
has the form 
$$\begin{pmatrix}
\mathscr{M}_0 &  \ast & \cdots &  \ast \\
0 & 1! \mathscr{M}_0 & \ddots& \vdots \\
 \vdots  & \ddots& \ddots &  \ast \\
0 & \cdots & 0 & n! \mathscr{M}_0
 \end{pmatrix},$$
 whence the first statement. 
 Here we identify the elements 
$\{x^{i},x^{j} \}=x^{i}_{(0)}x^{j}$ of  
$R$ with elements of $J_n R$ 
through the embedding $R\hookrightarrow J_n R$.
 The independence of the choice of generators follows from \cite{Van01}. 
\end{proof}

For $x \in J_n X$, let 
${\rm rk}\, x$ to be $1/(n+1) \times {\rm rank}\,\mathscr{M}_n (x)$. 
By Lemma~\ref{Lem:crucial},
 ${\rm rk}\,x$ is a non-negative integer.
 Moreover, for $x \in J_{\infty}X$, 
$ {\rm rk}\, \pi_{\infty,n}^{X}(x)$ does not depend on 
$n$. So we can define ${\rm rk}\, x$ to be this number. 
Lemma \ref{Lem:crucial} says that ${\rm rk}\, x$ 
is nothing but the rank of the matrix $\mathscr{M}_0$ at $x_0:=\pi_{\infty,0}^{X}(x)$.

Let $\LL$ be a chiral Poisson subscheme of $J_{\infty}X$.
We  define the  {\em rank stratification} of $\LL$ as follows. 
For $j \in \Z_{\ge 0}$, set:
\begin{align*}
\LL^0_j := \{x \in \LL \mid {\rm rk}\, x= j \}\subset 
\LL_j := \{x \in \LL \mid {\rm rk}\,x \le j \},
\end{align*}
Also,
set 
$\bar{\LL}=\pi_{\infty,0}^X(\LL)\subset X$,
and put
\begin{align*}
\bar{\LL}^0_j := \{x \in \bar{\LL} \mid {\rm rank}\,\mathscr{M}_0 (x)= j \}\subset 
\bar{\LL}_j := \{x \in \bar{\LL} \mid {\rm rank}\,\mathscr{M}_0(x)\le j \}.
\end{align*}
For $\LL=J_{\infty} X$, we have $\bar{\LL}=X$ and 
$$X= \bigsqcup_j \bar{\LL}^0_j$$
is precisely the rank stratification of $X$ defined by Brown and Gordon \cite{Brown-Gordon}. 
Note that $\LL_j=(\pi_{\infty,0}^X)^{-1}(\bar \LL_j)$ 
and $\LL_j^0=(\pi_{\infty,0}^X)^{-1}(\bar \LL_j^0)$ by definition.
\begin{Lem} \label{Lem:rank-strata}
\begin{enumerate}
\item $\LL_j$ 
is a closed subset of $\LL$ 
with 
$\LL_0 \subseteq \LL_1 \subseteq \cdots 
\subseteq \LL_d =\LL$ 
for some $d\in \Z_{\ge 0}$.
\item $\LL_j$ is a chiral Poisson subscheme of $J_{\infty} X$. 
\end{enumerate}
\end{Lem}
\begin{proof}
Part (1)  is clear by Lemma \ref{Lem:crucial} and \cite[Lemma 3.1 (1)]{Brown-Gordon}. 

(2) 
Let $\mathscr{I}_j$ be the defining ideal of $\bar{\LL}_j$. 
By \cite[Lemma 3.1 (2)]{Brown-Gordon}, $\mathscr{I}_j$ is a Poisson ideal of $R$. 
On the other hand, observe that for $n \ge 0$, 
$$(\pi_{n,0}^{X})^{-1}(\bar{\LL}_j) 
\cong \bar{\LL}_j \times_{X} J_n X.$$  Hence the defining 
ideal of $\LL_j=(\pi_{n,0}^{X})^{-1}(\bar{\LL}_j)$ 
is $\mathscr{I}_j {\otimes}_R J_\infty R$. 
So it is enough to show that $\mathscr{I}_j {\otimes}_R J_\infty R$ is chiral Poisson.

Let $u =\sum_i b_i c_i \in \mathscr{I}_j {\otimes}_R J_\infty R$, with $b_i \in 
\mathscr{I}_j, c_i \in J_\infty R$, 
Then by \eqref{eq:PVA-4} and Theorem \ref{Th:Poisson-structure-arc-spaces}, 
for all  $a \in R$ and $k \ge 0$, 
\begin{align*}
a_{(k)} u = \sum_i \left( (a_{(k)} b_i) \cdot c_i + b_i \cdot (a_{(k)} c_i )\right) 
= \sum_i \left( \delta_{k,0} \{ a, b_i\} \cdot c_i + b_i \cdot (a_{(k)} c_i )\right) 
\end{align*}
is in $\mathscr{I}_j {\otimes}_R J_\infty R$ 
since $\mathscr{I}_j$ is a Poisson ideal of $R$. 
So $\mathscr{I}_j {\otimes}_R J_\infty R$ is a chiral Poisson ideal by Lemma 
\ref{Lem:chiral Poisson-def}. 

\end{proof}

%
%
%

\section{Chiral Poisson cores and chiral symplectic cores} 
\label{sec:chiral Poisson-cores} 
Let {$X=\Spec R$ } be a reduced Poisson scheme. 
For $I$ an ideal of $R$, 
the {\em Poisson core} of $I$ is the biggest 
Poisson ideal contained in $I$. 
We denote it by $\P_R(I)$. 
The {\em symplectic core} $\Core_R(x)$ of a point $x \in X$  is  
the equivalence class of $x$ for $\sim$, with 
$$x \sim y \iff \P_R(\m_x) =\P_R(\m_y),$$ 
where $\m_x$ denotes the maximal ideal of $R$ corresponding to $x$. 
We refer the reader to \cite{Brown-Gordon} for more 
details about Poisson cores and symplectic cores.  

The aim of this section is to define analogue notions 
in the setting of vertex Poisson algebras. 

Let $V$ be a vertex Poisson algebra. 
By {\em ideal of $V$} 
we mean an ideal of $V$ in the associative sense. 
We will always specify {\em 
vertex Poisson ideal} or {\em chiral Poisson ideal} (see Section \ref{sec:PVA}) 
if necessary.
An ideal $I$ of $V$ is said to be {\em prime} if it is prime in the associative sense.

\begin{Def} \label{Def:chiral Poisson-core}
The {\em chiral Poisson core} of an ideal $I$ of $V$ is 
the biggest chiral Poisson ideal of $V$ 
contained in $I$. It exists since the sum of
two chiral Poisson ideals is chiral Poisson. We denote the chiral Poisson core of $I$ by 
$\P_V(I)$. 
\end{Def}



\begin{Lem} \label{Lem:Poisson-core-description} 
Let $I$ be an ideal of $V$. 
\begin{enumerate}
\item $\P_V(I)= \{x \in I \mid a^1_{(n_1)}\ldots 
a^k_{(n_k)} x \in I \text{ for all } a^i\in V,\, k \ge 0, 
\, n_i \ge 0\}$. 
\item {\rm{(}}\cite[Lemma 3.3.2(ii)]{Dix96}{\rm{)}} If $I$ is prime, then $\P_V(I)$ is prime. 
\item If $I$ is radical, then $\P_V(I)$ is radical. 
\end{enumerate}
\end{Lem}
\begin{proof}
Set 
$J =\{x \in I \mid a^1_{(n_1)}\ldots 
a^k_{(n_k)} x \in I \text{ for all } a^i\in V, \, k \ge 0, 
\, n_i \ge 0\}$. 

(1) 
By construction, $J \subset I$ and 
$J$ is chiral Poisson. 
Hence $J\subset \P_V(I)$. 
But if $K$ is a chiral Poisson ideal of $V$ 
contained in $I$, then for all $x \in K$, $a\in V$ and $n\geq 0$,
$a_{(n)}x \in K \subset I$, 
whence $x \in J$.
In conclusion, $J=\P_V(I)$. 


(3) Assume that $I$ is radical. Since $\P_V(I)$ is chiral Poisson, $\sqrt{\P_V(I)}$ 
is chiral Poisson as well by Lemma \ref{Lem:Poisson-ideal}, 
and it is contained in $\sqrt{I}$ since $\P_V(I)$ is contained in $I$.  
Hence, $$\sqrt{\P_V(I)} \subset \P_V(\sqrt{I})=\P_V(I).$$ 
But clearly $\P_V(I) \subset \sqrt{\P_V(I)}$, whence the equality 
$\sqrt{\P_V(I)}=\P_V(I)$ and the statement. 
\end{proof}

\begin{Co} \label{Co:minimal-prime-vertex-ideals}
Assume that there are finitely many minimal 
prime ideals $\p_1,\ldots,\p_r$ over $I$, that is, 
$I=\p_1\cap \ldots \cap \p_r$, and the 
prime ideals $\p_1,\ldots,\p_r$ are minimal.  
If $I$ is chiral Poisson, then so are the prime ideals $\p_1,\ldots,\p_r$. 
\end{Co}

\begin{proof} 
If $I$ is chiral Poisson, 
then $I \subset \P_V(\p_i) \subset \p_i$ for all $i$. 
But by Lemma \ref{Lem:Poisson-core-description} (2), 
the ideals $\P_V(\p_i)$, $i=1,\ldots,r$, are all prime. 
By minimality of the prime ideals $\p_i$ we deduce that 
$\p_i =\P_V(\p_i)$ for all $i$. 
In particular, the prime ideals $\p_i$, 
$i=1,\ldots,r$, are all chiral Poisson. 
\end{proof}

Set 
$$\LL := \Specm (V).$$ 
We define a relation $\sim$ 
on $\LL$ by
\begin{align*}
x \sim y \iff \P_V(\m_x) =\P_V(\m_y),
\end{align*}
where $\m_x$ is the maximal ideal corresponding to $x\in \LL$.

Clearly $\sim$ is an equivalence relation. 
We denote the equivalence class in $\LL$ 
of $x$ by $\Core_{\LL}(x)$, 
so that
$$\LL= \bigsqcup_{x} \Core_{\LL}(x).$$ 
We call the set $\Core_{\LL}(x)$ the {\em chiral symplectic core of $x$ 
in $\LL$}.

For $I$ an ideal of $V$, we denote by $\V(I)$
the corresponding zero locus in $\Specm V$, that 
is, 
$$\V(I)= \{x \in \LL \mid 
f(x)=0 \text{ for all } f \in I\} =\{x \in \LL \mid \m_x \supset I\}
,$$ 
and by $\widetilde \V (I)$ 
the corresponding 
closed scheme for the Zariski topology, that is, 
$$\widetilde \V(I)=\{\p \in \Spec V \mid \p \supset I\}.$$


\begin{Lem} \label{Lem:defining-ideal}
Let $x \in\LL$. 
Then $\widetilde \V( \P_V(\m_x))$ is the smallest 
chiral Poisson scheme 
containing $\overline{\Core_{\LL}(x)}$. 
Moreover, it is reduced and irreducible. 
\end{Lem}

\begin{proof} 
First of all, since $\P_V(\m_x)$ is a 
chiral Poisson, prime and radical ideal of $V$ by 
Lemma \ref{Lem:Poisson-core-description}, 
$\widetilde \V( \P_V(\m_x))$ is a reduced 
irreducible (closed) chiral Poisson subscheme of $\LL$.  
For any $y \in  \Core_{\LL}(x)$, we have  
$\m_y \supset \P_V(\m_y) =\P_V(\m_x)$. 
Hence 
$\overline{\Core_{\LL}(x)}\subset \widetilde \V( \P_V(\m_x))$. 
Next, if $I$ is a chiral Poisson ideal of $V$ such 
that $\overline{\Core_{\LL}(x)}\subset \widetilde \V(I)$, 
then in particular $\m_x \supset I$. Since $I$ is chiral Poisson, 
we get that $\m_x \supset \P_V(\m_x) \supset I$ by maximality of $\P_V(\m_x)$. 
Hence $\overline{\Core_{\LL}(x)} 
\subset \widetilde\V( \P_V(\m_x)) \subset \widetilde\V(I)$. 
This proves the statement. 
\end{proof}

\begin{Lem} \label{Lem:quotient} 
Let $\LL'$ be a reduced closed vertex Poisson subscheme of $\LL$. 
Then for any $x \in \LL'$, we have $\Core_{\LL'}(x)= \Core_{\LL}(x)$. 
\end{Lem} 
\begin{proof} 
Since $\LL'$ is a reduced closed vertex Poisson subscheme of $\LL$, 
$\LL' =\Specm (V/I)$, where $I$ is a vertex Poisson ideal of $V$. 
The maximal ideals of $V/I$ are precisely the 
quotients $\m/I$, where $\m$ is a maximal ideal of $V$ 
containing $I$, 
and $\P_{V/I}(\m/I) = \P_V(\m)/I$. 

Hence for $x \in \LL'$, we have 
\begin{align*}
\Core_{\LL'}(x) & =\{ y \in \LL' \mid \P_{V/I}(\m_y/I) =\P_{V/I}(\m_x/I) \} &\\
&=\{ y \in \LL' \mid \P_{V}(\m_y)/I=\P_{V}(\m_x)/I \} 
= \{ y \in \LL' \mid \P_{V}(\m_y)=\P_{V}(\m_x) \} ,&
\end{align*}
where $\m_x$ 
is the maximal ideal of $V$ corresponding to $x$. 
The maximal ideal $\m_x$ contains $I$ because $x \in  \LL'$.
On the other hand, 
\begin{align*}
\Core_{\LL}(x)=\{y \in \LL \mid \P_{V}(\m_y)=\P_{V}(\m_x) \}. 
\end{align*}
But if $ \P_{V}(\m_y)=\P_{V}(\m_x)$ for some $y \in \LL$, then 
$\m_y \supset  \P_{V}(\m_y)=\P_{V}(\m_x) \supset I$ because 
$I$ is a chiral Poisson ideal of 
$V$ contained in $\m_x$. 
This shows that if $y \in \Core_{\LL}(x)$, then $y \in \LL'$. 
Therefore 
\begin{align*}
\Core_{\LL}(x)=\{y \in \LL' \mid \P_{V}(\m_y)=\P_{V}(\m_x) \} =
\Core_{\LL'}(x).
\end{align*}
\end{proof}

\begin{Pro} \label{Pro:Poisson_center} 
Let $z \in \mathcal{Z}(V)$ and $x \in \LL$. 
Then $z$ is constant on $\V(\P_V(\m_x))$ 
and so on $\overline{\Core_{\LL}(x)}$.  
\end{Pro}

\begin{proof} 
Let $y \in\V(\P_V(\m_x))$, and let $\chi_x,\chi_y$ be the homomorphisms 
$\chi_x \colon V\to \C$,  $\chi_y \colon V\to \C$, corresponding to the 
maximal ideals $\m_x,\m_y$. 
It is enough to show that $\chi_x(z)=\chi_y(z)$. 
Set $\lambda:= \chi_x(z)$. Then $z- \lambda \in \ker \chi_x = \m_x$. 
In addition since $z$ is in the center, so is $z-\lam$, and then $a_{(n)} (z-\lambda)=0$ 
for any $a \in V$ and $n \ge 0$.    
Therefore $z-\lam \in \P_V(\m_x) \subset \m_y$, whence $\chi_y(z)=\lam$.  
By Lemma \ref{Lem:defining-ideal}, we conclude that 
$z$ is constant on $\overline{\Core_{\LL}(x)}$. 
\end{proof}

\begin{Lem} \label{Lem:core-in-arc}
Let $X=\Specm R$ be a reduced Poisson scheme, 
and 
let $\LL$ be a closed vertex Poisson subscheme of $\J X$. 
Set $\bar \LL= \pi^X_{\infty,0}(\LL)$. Let $x \in \LL$.  
\begin{enumerate}
\item  We have $\pi_{\infty,0}^X(\Core_{\LL}(x)) 
\subset \Core_{\bar{\LL}}(\pi_{\infty,0}^X(x)).$   
\item If ${\rm rk}\,x=j$,
then 
$\Core_{\LL}(x) \subset {\LL}_j^{0}.$ 
\end{enumerate}
\end{Lem}

\begin{proof}
(1) We have 
$\LL= \Specm \J R/I$, with $I$ a vertex Poisson ideal of $\J R$. 

Let $x \in \LL$. Let us first show that: 
\begin{align} \label{eq:poisson-cores-I}
\P_{R/(I\cap R)}(\m_x \cap R/(I\cap R)) = \P_{\J R/I}(\m_x/I) \cap R/(I\cap R).
\end{align}
The inclusion  $\P_{\J R/I}(\m_x/I) \cap R/(I\cap R)
\subset \P_{R/(I\cap R)}(\m_x \cap R/(I\cap R))$ 
is clear because the left-hand side  
is Poisson and is contained in $\m_x \cap R/(I \cap R)$. 
For the converse inclusion, 
let $a \in R$, $k\in \Z_{\ge 0}$ and $b \in \P_{R/(I\cap R)}(\m_x \cap R/(I\cap R))$. 
Then by Theorem \ref{Th:Poisson-structure-arc-spaces} 
$$a_{(k)} (b + I\cap R) =\delta_{k,0}(\{a,b\} + \{a,I\cap R\}) 
\in  \P_{R/(I\cap R)}(\m_x \cap R/(I\cap R))$$ 
since $\P_{R/(I\cap R)}(\m_x \cap R/(I\cap R))$ 
and $I\cap R$ are Poisson.  
By Lemma~\ref{Lem:chiral Poisson-def}, 
this shows that $\P_{R/(I\cap R)}(\m_x \cap R/(I\cap R))$ 
is a chiral Poisson ideal of $\J R/I$, contained 
in $\m_x/I$, whence the expected equality \eqref{eq:poisson-cores-I}. 

Let now $y \in \Core_{\LL}(x)$. 
Then $\P_{\J R/I}(\m_y/I)= \P_{\J R/I}(\m_x/I)$. 
From \eqref{eq:poisson-cores-I}, we deduce that 
$$\P_{R/(I\cap R)}(\m_y \cap R/(I\cap R))
=\P_{R/(I\cap R)}(\m_x \cap R/(I\cap R)),$$
and so 
$$\pi_{\infty,0}^X(y) \in 
\Core_{\bar{\LL}} (\pi_{\infty,0}^X(x))$$
since $\m_y \cap R/(I\cap R)$ is the maximal ideal of $R/(I\cap R)$ corresponding to 
$\pi_{\infty,0}^X(y) \in \bar{\LL}$. 
This proves the statement. 

(2)
By  Lemma \ref{Lem:crucial}, 
${\LL}_j^{0} = (\pi_{\infty,0}^X)^{-1}(\overline {\LL}_j^0)$. 
On the other hand, by \cite[Proposition 3.6]{Brown-Gordon},
$\Core_{\bar{\LL}}(\pi_{\infty,0}^X(x)) 
\subset \overline {\LL}_j^0$. 
Hence by (1), 
$$\Core_{\LL}(x) \subset (\pi_{\infty,0}^X)^{-1} (\Core_{\bar{\LL}}(\pi_{\infty,0}^X(x))) 
\subset (\pi_{\infty,0}^X)^{-1}(\overline {\LL}_j^0) 
= {\LL}_j^{0}.$$
\end{proof}
%

By Lemma \ref{Lem:core-in-arc} 
and Lemma \ref{Lem:crucial}, note that the stratification 
by chiral symplectic cores is a refinement of  
the rank stratification.

\section{$n$-chiral Poisson cores in $n$-th jet schemes} \label{sec:chiral Poisson-arc}
Let $R$ be a Poisson algebra,
 $n \in \Z_{\ge 0}$. 
Recall that the derivations $a_{(k)}$, $k \ge 0$, of $\J R$ acts on $J_n R$ 
 by \eqref{eq:arc-jets} and \eqref{eq:10-Ara12}. 
We say that an ideal $I$ of $J_n R$ is {\em $n$-chiral Poisson} 
if $a_{(k)} I \subset I$ for any $a \in R$ and any $k=0,\ldots,n$.

%

\begin{Lem} \label{Lem:jet-chiral}
For a Poisson ideal $I$ of $R$,
$J_n I$ is $n$-chiral Poisson.
\end{Lem}
\begin{proof}
This follows from \eqref{eq:10-Ara12} since $I$ 
is a Poisson ideal of $R$. 
\end{proof}

If $I$ is an ideal of $J_n R$, we define the 
{\em $n$-chiral Poisson core} of $I$ to be the biggest 
$n$-chiral Poisson ideal contained in $I$. We denote it 
by $\P_{J_n R}(I)$. 

Set $X:=\Specm R$. 
We define a relation $\sim$ 
on $J_n X$ by:
\begin{align*}
x \sim y \iff \P_{J_n R}(\m_x) =\P_{J_n R}(\m_y), 
\end{align*}
where $\m_x$ denotes the maximal ideal of $J_n R$ 
corresponding to $x$. 

Clearly $\sim$ is an equivalence relation. 
We denote the equivalence class in $J_n X$ 
of $x$ by $\Core_{J_n X}(x)$, 
so that
$$J_n X= \bigsqcup_{x} \Core_{J_n X}(x).$$ 
We call the set $\Core_{J_n X}(x)$ the {\em $n$-chiral symplectic core of $x$ 
in $J_n X$}.

Similarly to the case of chiral Poisson cores in a vertex Poisson algebra, we 
obtain the following facts: 
\begin{Lem}\label{Lem:properties-m-cores}  \begin{enumerate} 
\item Let $I$ be an ideal of $J_n R$.
If $I$ is prime (resp.~radical) then $\P_{J_n R}(I)$ is prime 
(resp.~radical). 
\item Let $I$ be an ideal of $J_n R$.  
If $I$ is $n$-chiral Poisson, then so are the minimal 
prime ideals over $I$.  
\item Let $x \in J_n X$. 
Then $\widetilde \V( \P_{J_n R}(\m_x))$ is the smallest 
$n$-chiral Poisson subscheme of $J_n X$ 
containing $\overline{\Core_{J_n X}(x)}$. 
\item Let $Y$ be a reduced Poisson subscheme of $X$, and let 
$x \in J_n Y$. 
Then $\Core_{J_n X}(x)=\Core_{J_n Y}(x)$ 
and $\pi_{n,0}^X(\Core_{J_n X}(x)) 
\subset \Core_{Y}(\pi_{n,0}^X(x)).$  
\end{enumerate}
\end{Lem}

\begin{Lem} \label{Lem:rank_der} 
Let $Y$ be a reduced $n$-chiral Poisson subscheme of $J_n X$, and 
$y \in Y$. Let $\{x^1,\dots, x^r\}$ be a generating set for $R$, 
and consider the matrix $\mathscr{M}_n(y)$ as 
in Section \ref{sec:Rank}. 
Suppose that the matrix $\mathscr{M}_n(y)$
has maximal rank $(n+1)r$.  
Then the tangent space at $y$ of $Y$ has dimension 
at least $(n+1)r$. Moreover, $Y$ has dimension 
at least $(n+1)r$. 
\end{Lem}
\begin{proof}
Since $Y$ is a reduced $n$-chiral Poisson subscheme of $J_n X$, 
$Y=\Spec J_n R/I$, where $I$ is a reduced $n$-chiral Poisson of $J_n R$. 
Moreover, $I \subset \m_y$ since $y\in Y$,  
where $\m_y$ denotes the maximal ideal of $J_n R$ 
corresponding to $y$. 

The hypothesis implies that the derivations $x^{i}_{(k)}$, $i=1,\ldots,r$, 
$k=0,\ldots,n$, are linearly independent in Der$(\mc{O}_{J_nX,y},\C)$. 
Indeed, if for some $\lam^{i}_{(k)}$, $i=1,\ldots,r$, 
$k=0,\ldots,n$, 
$$\sum_{i=1}^r \sum_{k=0}^n \lam^{i}_{(k)}x^{i}_{(k)}=0 
\quad\text{ in }\quad {\rm Der}(\mc{O}_{J_nX,y},\C),$$ 
then 
$$\sum_{i=1}^r \sum_{k=0}^n \lam^{i}_{(k)} (x^{i}_{(k)}(T^{l}x^{j})
+\m_y) =0\quad\text{ for all }\;  j=1,\ldots,r, \, l=0,\ldots,r,$$
and so $\lam^{i}_{(k)}=0$ for all $i=1,\ldots,r$ and $k=0,\ldots,n$
since the matrix $\mathscr{M}_n(y)$
has rank $(n+1)r$. 

Since $I$ is $n$-chiral Poisson and is contained in $\m_y$, 
we get that for all $i=1,\ldots,r$ and $k=0,\ldots,n$, $x^{i}_{(k)} (I) \subset I \subset \m_y$. 
Hence, the derivations $x^{i}_{(k)}$, $i=1,\ldots,r$, 
$k=0,\ldots,n$ are also linearly independent in Der$(\mc{O}_{Y,y},\C)$ 
since $\mc{O}_{Y,y} = \mc{O}_{J_nX,y}/(\mc{O}_{J_nX,y}\cap I)$. 
This shows that the tangent space at $y$ of $Y$ has dimension 
at least $(n+1)r$. 
 
The set of points $y \in Y$ such that matrix $\mathscr{M}_n(y)$
has maximal rank $(n+1)r$ is a nonempty open subset of $Y$. 
Hence it meets the set of smooth points of $Y$. 
By the first step, we deduce that for some smooth point $y \in Y$, 
the tangent space $T_y Y$ has dimension at least $(n+1)r$. 
Therefore $Y$ has dimension at least $(n+1)r$.
\end{proof}

Recall that $\iota_n$ (resp.~$\iota_\infty$) 
denotes the canonical 
embedding from $X$ to $J_n X$ (resp.~$\J X$). 
For $x$ in $X$, we simply denote by $x_n$ (resp.~$x_\infty$) 
 the element $\iota_n(x)$ (resp.~$\iota_\infty(x)$).

\begin{Pro} \label{Pro:equality-cores} 
Let $x \in X$,  
and set $Y:=\overline{\Core_{X}(x)}$. 
\begin{enumerate}
\item For any  $n \in \Z_{\ge 0}$, 
$$\overline{(\pi_{n,0}^{Y})^{-1}(Y_{reg})} = 
\V(\P_{J_n R}(\m_{x_n})).$$  
In particular, if $J_n Y$ is irreducible, then 
$$J_n Y= \V(\P_{J_n R}(\m_{x_n})).$$ 
\item We have: 
$$\J Y = \V(\P_{\J R}(\m_{x_\infty}))$$
\end{enumerate}
\end{Pro}

\begin{proof}
(1) By Lemma \ref{Lem:properties-m-cores} (4), 
$\Core_{J_n X}(x_n)=\Core_{J_n Y}(x_n)$, 
and 
$$\Core_{J_n X}(x_n) \subset (\pi_{n,0}^{Y})^{-1}
(\Core_{X}(x)) \subset (\pi_{n,0}^{Y})^{-1}
(Y_{reg})$$ since by \cite[Lemma 3.3 (2)]{Brown-Gordon}, 
$\Core_{X}(x)$ is smooth in its closure. 
Hence
$$\overline{\Core_{J_n X}(x_n)} \subset 
\overline{(\pi_{n,0}^{Y})^{-1}
(Y_{reg})} \subset J_n Y.$$

Since $\overline{(\pi_{n,0}^{Y})^{-1}
(Y_{reg})}$ is an irreducible component of 
$J_n Y$  (cf.~Section \ref{sec:arc}) and since $J_n Y$ 
is $n$-chiral Poisson, $\overline{(\pi_{n,0}^{Y})^{-1}
(Y_{reg})}$ is an $n$-chiral Poisson subscheme 
of $J_n X$. 
Hence by Lemma~\ref{Lem:defining-ideal}, 
$$\overline{\Core_{J_n X}(x_n)} \subset \V(\P_{J_n A}(\m_{x_n})) 
\subset \overline{(\pi_{n,0}^{Y})^{-1}(Y_{reg})},$$ 
where $A = \C[Y]$. 
Let $x^{1},\ldots,x^{r}$ be generators of $A$,  
where $r=\dim Y$. 
By \cite[Proposition 3.6]{Brown-Gordon} 
(proof of (2)), the matrix $\mathscr{M}_0$ 
has maximal rank $r$ at $\m_x$.  
Hence, by Lemma \ref{Lem:crucial}, $\mathscr{M}_n$ 
has rank $(n+1)r$ at $\m_{x_n}$. Since 
$\P_{J_n A}(\m_{x_n})$ is an $n$-chiral Poisson 
ideal of $J_n A$, 
it results from Lemma \ref{Lem:rank_der} 
that $\V(\P_{J_n A}(\m_{x_n})))$ 
has dimension at least $(n+1)r$. 
But 
$$\dim \overline{(\pi_{n,0}^{Y})^{-1}
(Y_{reg})}=(n+1)r.$$
Both $\V(\P_{J_n A}(\m_{x_n})))$ 
and $\overline{(\pi_{n,0}^{Y})^{-1}
(Y_{reg})}$ are closed and irreducible, whence the first assertion of (1). 
Indeed note that $\V(\P_{J_n A}(\m_{x_n}))) 
=\V(\P_{J_n R}(\m_{x_n})))$ 
since $\Core_{J_n X}(x_n)=\Core_{J_n Y}(x_n)$. 

The second assertion follows from the fact that 
$\overline{(\pi_{n,0}^{Y})^{-1}(Y_{reg})}$ is an irreducible 
component of $J_n Y$. 

Part (2) follows from part (1) and Kolchin's Theorem~\ref{Th:Kolchin}. 
\end{proof}

\section{Partial stratification by chiral symplectic leaves} \label{sec:chiral-symplectic-leaves}

Recall that there is a well-defined stratification of 
$X$ by symplectic leaves \cite{Brown-Gordon}. 

We assume in this section that 
the Poisson bracket on $X={\rm Specm}\,R$ is 
{\em algebraic}, that is, the symplectic leaves in $X$ 
are all locally closed. 
Then  \cite[Proposition 
3.6]{Brown-Gordon} the symplectic leaves coincide
with the symplectic cores of $X$,  
and the defining ideal of the symplectic core closure 
of a point $x\in X$ is $\P_R(\m_x)$. 

Let $x_0 \in X$, and  $\mathscr{L}_X(x_0)$  the symplectic 
leaf through $x_0$ in $X$. 
Set 
$$Y =\overline{\mathscr{L}_X(x_0)} \quad \text{ and } \quad A=\C[Y].$$ 
Note that $(\pi_{n,0}^Y)^{-1}(\mathscr{L}_X(x_0))$ 
is open in $J_n Y$ since 
$\mathscr{L}_X(x_0)$ is open in its closure $Y$. 
Moreover,  $(\pi_{n,0}^Y)^{-1}(\mathscr{L}_X(x_0)) 
=J_n \mathscr{L}_X(x_0)$ by \cite[Lemma 2.3]{EinMus}. In particular, 
one can equip
$(\pi_{n,0}^Y)^{-1}(\mathscr{L}_X(x_0))$ with the structure of a smooth analytic 
variety. 

Let $x  \in J_n \mathscr{L}_X(x_0)$. 
We 
denote by $\mathscr{L}_{J_n X}(x)$ the set 
of all $y \in J_n \mathscr{L}_X(x_0)$ 
which can be reach from $x$ by traveling along 
the integral flows of vector fields 
$a_{(k)}$, $a \in A$, $k=0,\ldots,n$.  

We call $\mathscr{L}_{J_n X}(x)$ the {\em $n$-chiral 
symplectic leaf of $x$ in $J_n X$}. 
Note that we defined $n$-chiral 
symplectic leaves 
only for elements 
in $\bigcup_{x' \in X} J_n \mathscr{L}_X(x')$ 
which is a priori different from $J_n X$. 


\begin{Lem}  \label{lem:leaf}
Let $x  \in J_n \mathscr{L}_X(x_0)$. Then the defining ideal 
of the Zariski closure of $\mathscr{L}_{J_n X}(x)$ 
is $\P_{J_n R}(\m_{x})$. 
\end{Lem}

\begin{proof} 
First of all, by construction of $\mathscr{L}_{J_n X}(x)$, 
we have $\mathscr{L}_{J_n X}(x) \subset J_n Y$. 
So by Lemma \ref{Lem:properties-m-cores} (4), 
$\P_{J_n R}(\m_{x})=\P_{J_n A}(\m_{x})$. 

We now follow the ideas of the proof of \cite[Lemma 3.5]{Brown-Gordon}. 
Let $\mathscr{K}_{x}$ be the defining ideal of 
 $\overline{\mathscr{L}_{J_n X}(x)}$. 
 
We first show that $\P_{J_n A}(\m_{x}) \subset  \mathscr{K}_{x}$. 
Let $$\widetilde{J_n Y} := \Specm \widetilde{J_n A}, 
\quad  \text{ with }\quad 
\widetilde{J_n A} = J_n A/\P_{J_n A}(\m_{x}),$$ 
and denote by $\widetilde a$ the image of $a \in J_n A$ 
in $\widetilde{J_n A}$. 
For $r >0$, $B(r)$ denotes the open complex analytic disc of radius $r$. 
Let $a \in A$ and $k \in \{0,\ldots,n\}$.
Since $a_{(k)}\P_{J_n A}(\m_{x})\subset \P_{J_n A}(\m_{x})$,
$a_{(k)}$ defines a  derivation on $\widetilde{J_n A} $,
which we denote by  $\widetilde{a}_{(k)}$.
Consider $\sigma_{x} \colon B(r) \to J_n Y$ 
and  $\widetilde \sigma_{x} \colon B(r) \to \widetilde{J_n Y}$ be integral 
curves of the vector fields  $a_{(k)}$ and $\widetilde{a}_{(k)}$ respectively, 
with $\sigma_{x}(0) = x$, $\widetilde\sigma_{x}(0) =\widetilde x$. 

Viewing $\widetilde{J_n Y}$ as a subset of $J_n Y$, let us show that 
$\widetilde\sigma_{x}=\sigma_{x}$ in a neighborhood of~$0$. 
Let $f \in J_n A$. By definition of an integral curve, 
\begin{align} \label{eq:leaf} 
\dfrac{\d}{\d z} (f \circ \sigma_{x}) & = a_{(k)} (f) \circ \sigma_x, &\\ \label{eq:leaf2}
\dfrac{\d}{\d z} (\widetilde f \circ \widetilde \sigma_{x}) & = \widetilde{a}_{(k)} 
( \widetilde f) \circ \widetilde\sigma_x.
\end{align}
But the left hand side of \eqref{eq:leaf2} is  $\dfrac{\d}{\d z} ( f \circ \widetilde\sigma_{x})$, and the right hand side is 
${a_{(k)}} 
( f) \circ \widetilde\sigma_x$ because $\P_{J_n A}(\m_{x})$ is a chiral 
Poisson ideal of $J_n A$. 
Hence by \eqref{eq:leaf}, we conclude by the uniqueness of flows that 
$\widetilde \sigma_{x}=\sigma_{x}$ in a neighborhood of $0$. 
Since the $n$-chiral symplectic leaf $\mathscr{L}_{J_n X}(x)$ is 
by definition
obtained by traveling along integral curves to fields $a_{(k)}$, 
the chiral symplectic leaf $\mathscr{L}_{J_n X}(x)$, and so its closure, 
is contained in $\V(\P_{J_n A}(\m_{x}))$, 
whence 
$$\P_{J_n A}(\m_{x}) \subset  \mathscr{K}_{x} \subset \m_x.$$
 
To show the equality 
$\P_{J_n A}(\m_{x}) =  \mathscr{K}_{x}$, it remains to prove that 
$\mathscr{K}_{x}$ is an $n$-chiral Poisson ideal of $J_n A$. 

Let $f\in  \mathscr{K}_{x}$, $a \in A$ and $k \in \{0,\ldots,n\}$. 
Let $\sigma_{x} \colon B(r) \to J_n Y$ be an integral curve 
to $a_{(k)}$, with $\sigma_{x}(0)=x$. Then, by definition of an integral curve, 
\eqref{eq:leaf} holds.
On a complex analytic neighborhood of $x$, 
$f \circ \sigma_x =0$ since the image of $\sigma_x$ is in $\mathscr{L}_{J_n X}(x)$. 
Hence 
$$0=\dfrac{\d}{\d z} (f \circ \sigma_x)(0) =(a_{(k)} (f) \circ \sigma_x) (0)= 
a_{(k)} (f) (x).$$
As a consequence, $a_{(k)} (\mathscr{K}_x) \subset \m_{x_n}$ 
for all $a \in A$ and all $k \in \{0,\ldots,n\}$. 
Repeating this argument with $x$ replaced by each of the members
of $\mathscr{L}_{J_n X}(x)$, we conclude that 
$a_{(k)} (\mathscr{K}_x) \subset \mathscr{K}_x$ 
for all $a \in A$ and all $k \in \{0,\ldots,n\}$, that is, 
that $\mathscr{K}_x$ is $n$-chiral Poisson. 
\end{proof}

\begin{Co} \label{Co:leaf}
Let $x  \in J_n \mathscr{L}_X(x_0)$. 
Then the defining ideal 
of the closure of $\Core_{J_n X}(x)$ 
is $\P_{J_n R}(\m_{x})$. 
\end{Co}

\begin{proof}
By Lemma \ref{lem:leaf}, 
$\mathscr{L}_{J_n X} (x) \subset \Core_{J_n X}(x)$. 
Indeed, if $y \in \mathscr{L}_{J_n X} (x)$, 
then $\mathscr{L}_{J_n X} (y)=\mathscr{L}_{J_n X} (x)$ 
and so $\V(\P_{J_n R}(\m_{y}))=\V(\P_{J_n R}(\m_{x}))$ by Lemma \ref{lem:leaf},  
that is, $\P_{J_n R}(\m_{y})=\P_{J_n R}(\m_{x})$, whence 
$y \in \Core_{J_n X}(x)$. 
So by Lemma \ref{Lem:defining-ideal}, 
$$\V(\P_{J_n R}(\m_{x}))=\overline{\mathscr{L}_{J_n X} (x) } 
 \subset \overline{\Core_{J_n X}(x)} \subset \V(\P_{J_n R}(\m_{x})),$$
whence the statement. 
\end{proof}

%

\begin{Co} \label{Co:equality-cores}  
Let $x \in X$, and set $Y:=\overline{\Core_{X}(x)}$. 
Then 
$$\J Y = \overline{\Core_{\J X}(x_\infty)}.$$
Moreover, for $n \in \Z_{\ge 0}$, if $J_n Y$ is irreducible, 
then 
$J_n Y = \overline{\Core_{J_n X}(x_n)}.$ 
\end{Co}
Here, recall that $x_n$ (resp.~$x_\infty$) stands for $\iota_n (x)$ 
(resp.~$\iota_\infty(x)$) as explained before Proposition \ref{Pro:equality-cores}.

\begin{proof}
Recall that $\mathscr{L}_X(x)=\Core_{X}(x)$ 
is contained in the smooth locus of $Y$, and that 
$\overline{\pi_{n,0}^{-1}(\Core_X(x))}$ has dimension 
$(n+1) \dim Y$. 

Since $\P_R(\m_x)$ is the defining ideal of 
$\overline{\Core_X(x)}$, it results from 
Proposition \ref{Pro:equality-cores} that 
for all $n \in\Z_{\ge 0} \cup \{\infty\}$, 
$\V(\P_{J_n R}(\m_{x_n}))=\overline{(\pi_{n,0}^Y)^{-1}(
Y_{reg})}$, and we have $\V(\P_{J_n R}(\m_{x_n}))=J_n Y$ if $J_n Y$ is 
irreducible. In particular, $\J Y=\V(\P_{\J R}(\m_{x_\infty}))$. 

So by Corollary \ref{Co:leaf}, for all $n \in\Z_{\ge 0}$, 
$$\overline{\Core_{J_n X}(x_n)} = \V(\P_{J_n R}(\m_{x_n})) 
=\overline{(\pi_{n,0}^Y)^{-1}(
Y_{reg})} \subset J_n Y.$$
{Taking the limit when $n$ goes to $+\infty$}, 
we obtain:
$$\overline{\Core_{\J X}(x_\infty)} = \V(\P_{\J R}(\m_{x_\infty})) 
= \J Y.$$
This concludes the proof.  
\end{proof}
Now assume further 
that $X$ has only finitely many symplectic leaves. 
(Note that  the poisson bracket on $X$ is algebraic under this condition
by
\cite[Proposition 
3.6]{Brown-Gordon}.)
 If $X_1,\ldots,X_r$ are the irreducible components 
of $X$, then for some 
 $x_1,\ldots,x_r\in X$, we have 
$$X_i =\V(\P_{R}(\m_i))= \overline{\Core_{X}(x_i)}, \qquad 
i=1,\ldots,r,$$
where $\m_1,\ldots,\m_r$ are the maximal ideals of $R$ corresponding 
to $x_1,\ldots,x_r$, respectively, see \cite{Gin03}. 

From the decomposition 
$X=X_1 \cup \ldots \cup X_r$, we get  that
$$\J X= \J X_1 \cup \ldots\cup \J X_r$$
since the $X_i$ are closed (see 
Section \ref{sec:arc}). 
Moreover, $\J X_1, \ldots, \J X_r$ are precisely the irreducible 
components of $\J X$. Indeed, for $i=1,\ldots,r$, $\J X_i$ 
is closed in $\J X$ since $X_i$ is closed in $X$,  
and for any 
$i\not=j$, we have $\J X_i \not \subset \J X_j$, 
otherwise, taking the image by the canonical projection 
$\pi_{\infty,0} \colon \J X \to X$, 
we would get 
$X_i \subset X_j$.

Hence, as a consequence of Corollary \ref{Co:equality-cores}, we obtain the following result. 

\begin{Th} \label{Th:arc-symplectic-cores}
Let $X$ be a Poisson scheme.
Assume that $X$ has only finitely many symplectic leaves. 
Then each irreducible components of 
$\J X$ is the closure of some chiral symplectic core. 

More precisely, if $X_1,\ldots,X_r$ are the irreducible 
components of $X$,
then for $i=1,\ldots,r$,  
$X_i = \overline{\Core_{X}(x_i)}$  for some $x_i \in X_i$, 
and we have: 
$$\J X_i= \overline{\J \Core_{X}(x_i)} = 
\overline{\Core_{\J X}(x_{i,\infty})}.$$ 
\end{Th}

\section{Applications to quasi-lisse vertex algebras} 
\label{sec:quasi-lisse}

%
%
%
%
%

In this section we assume that $V$ is a vertex algebra (not necessarily commutative or Poisson).

Recall that $V$ is naturally filtered by the Li filtration (\cite{Li05}, see also
\cite{Ara12}), 
$$V= F^0 V \supset F^1 V \supset \cdots \supset F^p V\supset \cdots ,$$
where $F^p V$ is the subspace of $V$ spanned by the vectors 
$$a_{(-n_1-1)}^1\ldots a_{(-n_r-1)}^r b,$$ 
with $a^{i} \in V$, $b \in V$, $n_i \in\Z_{\ge 0}$, $n_1 +\cdots+n_r \ge p$. 
The associated graded vector space
$\gr V =\oplus_p F^p V/  F^{p+1} V$ is naturally a vertex Poisson algebra 
\cite{Li05}.
We have 
$$F^1 V=C_2(V):={\rm span}_\C\{a_{(-2)}b \mid a,b, \in V\}.$$
Let 
$$R_V= V/C_2(V) = F^{0} V/ F^1 V \subset \gr V$$ 
be the {\em Zhu $C_2$-algebra of $V$}. 
It is a Poisson algebra \cite{Zhu96}, and the Poisson algebra structure can be obtained 
by restriction to $R_V$ of the vertex Poisson algebra on $\gr V$. 
Namely,
$$1 = \overline{ \vac}, \quad \bar a \cdot \bar b =\overline{a_{(-1)}b}\quad \text{ and }\quad
\{\bar a,\bar b\}=\overline{a_{(0)}b},$$
for $a,b \in V$, 
where $\bar a = a + C_2(V)$. 

Let 
$$\tilde{X}_V:=\Spec(R_V)\quad \text{ and }\quad X_V :=\Specm(R_V)$$ 
be the {\em associated scheme} and the {\em associated variety} 
of $V$, respectively (\cite{Ara12}). 

\bigskip

{\em We assume that the filtration $(F^p V)_p$ is separated, that is, 
$\bigcap F^p V = \{0\}$ and that $V$ is strongly finitely generated, 
that is, $R_V$ is finitely generated.}
Note that the first condition is satisfied if $V$ is positively graded.

%
%
%
\begin{Th}[{\cite[Lemma 4.2]{Li05}}, 
{\cite[Proposition 2.5.1]{Ara12}}]
\label{thm:surj-poisson} 
The identity map $ R_V \rightarrow R_V$ induces a surjective vertex 
Poisson algebra homomorphism
\begin{align*}
J_{\infty} R_V =\C[J_{\infty} (\tilde X_V)]  \twoheadrightarrow \gr V.
\end{align*}
\end{Th}

The {\em singular support} of a vertex algebra $V$ is  
\begin{align*}
{\widetilde{SS}(V)}:= \Spec \gr V \subset J_\infty (\tilde X_V).
\end{align*}
{We set 
\begin{align*}
SS(V):= \Specm \gr V \subset J_\infty X_V.
\end{align*}
In the above inclusion, $ J_\infty X_V$ is viewed as a topological space.}



Recall from the introduction 
that the vertex algebra $V$ is called {\em quasi-lisse} if 
the Poisson variety $X_V$ has finitely many symplectic leaves 
(\cite{{Arakawam:kq}}).

\begin{Th}\label{Th:quasi-lisse}
Assume that $V$ is quasi-lisse. Then 
$SS(V)$ is a finite union of chiral symplectic  
cores closures in $\gr V$. 
Moreover, $SS(V) = \J X_V$ as topological spaces.
\end{Th}
\begin{proof}
Set 
$$\LL = \Specm \gr V=SS(V),$$
and let $X_1,\ldots,X_r$ be the irreducible components of $X_V$. 
By Theorem \ref{Th:arc-symplectic-cores}, we have 
\begin{align}\label{eq:SS_V}
\J X_V = \overline{\Core_{\J X_V}(x_{1,\infty})} \cup \ldots \cup 
\overline{\Core_{\J X_V}(x_{r,\infty})},
\end{align}
where $x_i \in (X_i)_{reg}$ for $i=1,\ldots,r$. 
By Theorem \ref{thm:surj-poisson}, 
$\gr V$ is a vertex Poisson algebra quotient of $\J R_V$, 
that is, $\gr(V)=\J R_V/I$ with $I$ a vertex Poisson ideal of 
$\J (R_V)$. 
{Furthermore, the surjective morphisms, 
$$\J R_V  \twoheadrightarrow \gr V  \twoheadrightarrow 
R_V,$$ 
induce injective morphisms of varieties, 
$$X_V \hookrightarrow \LL \hookrightarrow \J X_V,$$ 
and the composition map is $\iota_\infty$. 
Hence for $x \in X_V$, we get that $\m_\infty \supset 
I$, where $\m_\infty$ denotes the maximal ideal 
of $\J R_V$ corresponding to $x_\infty$, 
and so $x_\infty$ is a point of~$\LL$.} 

Therefore, by Lemma \ref{Lem:quotient}, 
$\Core_{\J (X_V)}(x_{i,\infty})=\Core_{\LL}(x_{i,\infty})$ 
for any $i=1,\ldots,r$. 
Then from \eqref{eq:SS_V} and Theorem \ref{thm:surj-poisson}, 
we obtain that 
$$\LL\subset 
\J X_V = \overline{\Core_{\LL}(x_{1,\infty})} \cup \ldots \cup 
\overline{\Core_{\LL}(x_{r,\infty})} \subset \LL,$$
since $\LL$ is closed, whence the first statement 
and the required equality $\LL = \J X_V$. 
\end{proof}

\begin{Co}
Suppose that 
$\tilde{X}_V$ is smooth, reduced  and  symplectic.
Then $\gr V$ is simple as a vertex Poisson  algebra,
and hence,  $V$ is simple. 
\end{Co}

\begin{proof}
If $X_V$ is a smooth symplectic variety then 
$J_{\infty}X_V$ consists of a single chiral symplectic core.
So $J_{\infty}X_V = \Core_{J_{\infty}X_V}(x)$ 
for any $x\in J_{\infty}X_V$. 
It follows that there is no nonzero proper chiral Poisson subscheme in 
$J_{\infty} X_V$. 
So by Theorem \ref{Th:quasi-lisse}, 
there is no nonzero proper chiral Poisson subscheme in 
$\Spec \gr V$, too. 
Hence $\gr V$ is simple as a vertex Poisson algebra. 
And so $V$ is simple since any vertex ideal $I \subset V$ 
defined a vertex Poisson, and so chiral Poisson, ideal $\gr I$ 
in $\gr V$. 
\end{proof}

For example, if $X$ is a smooth affine variety, then the global section 
of the chiral differential operators $\mc{D}_X^{ch}$  
(\cite{MalSchVai99,GorMalSch04,BeiDri04}) is simple, 
because its associated scheme is canonically isomorphic to $T^* X$.
In particular, 
the global section of the chiral differential operators $\mc{D}_{G,k}^{ch}$ on
a reductive group  $G$ 
(\cite{GorMalSch01,ArkGai02})
is simple at any level $k$.

\section{Adjoint quotient and arc space of Slodowy slice} 
\label{sec:adjoint-quotient}
Recall that $\g$ is a complex simple Lie algebra. 
Identify $\g$ with $\g^*$ through the Killing form $(~|~)$ of $\g$. 
Let $(e,h,f)$ be an $\sl_2$-triple, and $\Slo_f=f+\g^{e}$ the corresponding 
Slodowy slice, with $\g^{e}$ the centralizer of $e$ in $\g$. 
Recall that $\g^*\cong\g$ is a Poisson variety and that 
the symplectic leaves of $\g^*\cong \g$ are the (co)adjoint orbits. 
The algebra $R_f:=\C[\Slo_f]$ inherits a Poisson 
structure from $\C[\g]$ by Hamiltonian reduction \cite{GanGin02}. 
The Hamiltonian reduction can also be described in terms of the 
BRST cohomology, essentially following Kostant and Sternberg \cite{KosSte87}. 
The symplectic leaves of $\Slo_f$ are precisely the intersections of 
the adjoint orbits of $\g$ with $\Slo_f$.

Let $p_1,\ldots,p_\ell$ be homogeneous generators of $\C[\g^*]^G\cong\C[\g]^G\cong\C[\g]^\g$. 

Consider the adjoint quotient map
$$\psi \colon \g \to \g/\! / G \cong \C^\ell ,\quad x \mapsto (p_1(x),\ldots,p_\ell(x)),$$
and its restriction $\psi_f$ to $\Slo_f$, 
$$\psi_{f} \colon \Slo_f \to \g/\! / G \cong \C^\ell.$$ 
We first recall some facts about $\psi_f$ and its fibers (\cite{Pre02}).

The morphism $\psi_{f}$ is faithfully flat. As a consequence, 
$\psi_{f}$ is surjective and all fibers have the dimension $r-\ell$, 
where $r=\dim \g^{e}$. Furthermore the fibers of $\psi_f$ are generically smooth, 
that is, contain a smooth open dense subset of dimension $r-\ell$, and they are 
irreducible. 


\begin{Lem} \label{Lem:fiber-of-psi-finite}
Let $\xi \in \g/\!/ G$. Then $\psi_f^{-1}(\xi)$ is a finite union of symplectic leaves.
Hence it is the closure of some  symplectic leaf 
closure. 
\end{Lem}
\begin{proof}
The proof is standard, we recall it for the convenience 
of the reader.  

We first prove the statement for the morphism $\psi$. 
Let $x \in \psi^{-1}(\xi)$, and write $x=x_s+x_n$ its Jordan decomposition. 
Let $\g=\n_- \oplus\h+\n_+$ be a triangular 
decomposition of $\g$. One can assume that $x_s \in \h$, and that 
$x_n\in \n_+$ since $x_n \in \g^{x_s}$. 
Let now $y \in \psi^{-1}(\xi)$. Then $p_i(y)=p_i(x)$ for $i=1,\ldots,\ell$, 
and so $y_s$ is conjugate to $x_s$ by an element $s$ 
of the Weyl group $W(\g,\h)$ of $(\g,\h)$. 
Let $\tilde s$ be a Tits lifting of $s$ in $G$,
 so that 
$y=\tilde s(x_s + {\tilde s}^{-1} y_n)$. Since $W(\g,\h)$ is finite, 
it is sufficient to show that
there are only finitely many possible choices of $y_n$ up to 
conjugation by the centralizer $G^{y_s}$ of $y_s$ in  $G$.
However,
since $y_n\in \g^{y_s} \cap \mc{N}$ 
and
the set 
$\g^{y_s} \cap \mc{N}$ consists of a finite union of $(G^{y_s})^\circ$-orbits,
where $(G^{y_s})^\circ$ is the identity connected component 
subgroup of $G^{y_s}$, 
the assertion follows.
In conclusion, $ \psi^{-1}(\xi)$ is a finite union of $G$-orbits of ${\g}$, 
that is,  a finite union of symplectic leaves of $\g$. 
Since $ \psi^{-1}(\xi)$ is irreducible and $G$-invariant, we deduce 
that $ \psi^{-1}(\xi)$ is the closure of 
 some symplectic leaf. 

Next, $\psi_f^{-1}(\xi)=\psi^{-1}(\xi) \cap \Slo_f$ and the symplectic 
leaves of $\Slo_f$ are the intersections of adjoint orbits of $\g$ with $\Slo_f$. 
So $\psi_f^{-1}(\xi)$ is a finite union of symplectic leaves, too. 
As $\psi_f^{-1}(\xi)$ is irreducible \cite{Gin03}, 
$\psi_f^{-1}(\xi)$ is a symplectic leaf closure.
\end{proof}

By Kostant \cite{Kos78}, 
$$\psi_f^{-1}(0)=\Slo_f \cap \mc{N},$$
with $\mc{N}$ the nilpotent cone of $\g$. 


\begin{Pro} \label{Pro:zero-fiber}
Let $n \in \Z_{\ge 0}$. 
Then $(J_n \psi_f)^{-1}(0) =J_n( \psi_f^{-1}(0))$ is a reduced complete intersection,  
and it is irreducible. Moreover, $(J_\infty \psi_f)^{-1}(0) =J_\infty (\psi_f^{-1}(0))$ is irreducible and reduced.  
\end{Pro}

\begin{proof} 
The variety $\Slo_f \cap \mc{N}$
is normal, reduced and is a complete intersection  \cite{Gin08},
with rational singularities \cite[Lemma 3.1.2.1]{AKM}.
As $\psi_f^{-1}(0)=\Slo_f \cap \mc{N}$, 
it follows from the main results 
of \cite{Mus01} and its consequences that $J_n (\psi_f^{-1}(0))$ is also reduced, 
irreducible and a complete intersection for any $n\ge 0$. 
Now observe that $(J_n \psi_f)^{-1}(0) =J_n (\psi_f^{-1}(0))$ by the properties of 
jet schemes (cf.~Section \ref{sec:arc}). 
This proves the first part of the statement. 

Since $\psi_f^{-1}(0)$ is irreducible, $\J (\psi_f^{-1}(0))$ is irreducible 
by Theorem \ref{Th:Kolchin} 
and, from the above, we get that 
$(J_\infty \psi_f)^{-1}(0) =J_\infty( \psi_f^{-1}(0))$ is irreducible. 
Because all $J_n (\psi_f^{-1}(0))$ are reduced, $J_\infty( \psi_f^{-1}(0))$ 
is reduced, too. 
\end{proof}

Next, we wish to prove that the other fibers of $\J \psi_f$ 
are also reduced and irreducible. 
To this end, we use ideas of \cite[\S\S5.3 and 5.4]{Pre02}. 

The Slodowy slice has a contracting $\C^*$-action. 
Recall briefly the construction. 
The embedding $\haru_\C\{e,h,f\} 
\cong \sl_2 \hookrightarrow \g$ 
exponentiates to a homomorphism 
\hbox{$SL_2 \to {G}$}. By restriction to the one-dimensional 
torus consisting of diagonal matrices, we obtain 
a one-parameter subgroup $\tilde{\rho} \colon \C^* \to {G}$. 
Thus $\tilde{\rho}(t)x=t^{2j}x$ for any $x \in\g_j=\{y \in \g \mid 
[h,y]=2 j y\}$. 
For $t\in\C^*$ and $x\in\g$, set
\begin{eqnarray} \label{eq:rho}
{\rho}(t)x := t^{2}\tilde{\rho}(t)x.
\end{eqnarray}
So, for any $x\in\g_j$, ${\rho}(t)x =t^{2+2j}x$. In particular, 
${\rho}(t)f=f$ and 
the $\C^*$-action of ${\rho}$ stabilizes $\Slo_f$. 
Moreover, it is contracting to $f$ on $\Slo_f$, that is, 
$$\lim_{t\to 0} {\rho}(t)(f+x)=f$$ for any $x\in\g^{e}$. 
The $\C^*$-action $\rho$ induces a positive grading on $\Slo_f$, 
and so on $R_f =\C[\Slo_f] \cong \C[\g^{e}]$.  


Let $n \in \Z_{\ge 0} \cup \{\infty\}$. 
Similarly, we define a contracting $\C^*$-action on $J_n \Slo_f$ 
and a positive grading on $J_n R_f$ as follows. 

Let $x^{1},\ldots,x^{r}$ be a basis of $\g^{e}$ so that 
$$J_n R_f \cong J_n \C[\g^{e}] 
\cong \Spec \C [ x^i_{(-j-1)}  \, ;\,  
i=1,\ldots,r,\, j=0,\ldots,n].$$ 
One can assume that the $x^{i}$'s are Slodowy homogeneous. 
One defines a grading on $J_n R_f$ by setting 
$$\deg  x^i_{(- j -1)} = \deg x^{i} + j.$$ 
Since the grading is positive, it gives a contracting $\C^*$-action 
on  $J_n \Slo_f$. 
Indeed, consider the morphism 
$\C[J_n \Slo_f ] \to \C[J_n \Slo_f ] \otimes \C[t,t^{-1}]$, 
$f \mapsto f \otimes t^{\deg f}$, for homogeneous $f$. 
Its comomorphism 
induces a  
$\C^*$-action 
$$\mu_n \colon \C^* \times J_n \Slo_f \to J_n \Slo_f$$which is 
contracting since $\deg f\ge 0$ for any homogeneous $f$. 

The above grading gives an increasing filtration on 
$J_n R_f$ in an obvious way: 
$$\mathscr{F}_{p} (J_n R_f):= \oplus_{j \le p} (J_n R_f)_j,\qquad p \ge 0.$$
Given a quotient $M$ of $J_n R_f$, we define 
a filtration $(\mathscr{F}_{p} M)_p$ of $M$ by setting 
$$\mathscr{F}_{p} M := 
\tau_n (\mathscr{F}_p (J_n R_f)),$$ where $\tau_n$ is the canonical 
quotient morphism $\tau_n \colon J_n R_f \to M$. 
We denote by ${\gr } M$ } the corresponding graded space.

For $M$ a subspace of $J_n R_f$ denote by ${\rm gr} M$   
the homogeneous 
subspace of $J_n R_f$ with the property that 
$g \in {\rm gr} M \cap (J_n R_f)_p$ 
if and only if there is $\tilde g \in M$ such that 
$\tilde g -g \in \mathscr{F}_{p-1} (J_n R_f)$. 
Obviously the subspace ${\rm gr} M$ 
is invariant for the $\C^*$-action $\mu_n$. 
If $M$ is an ideal of $J_n R_f$ then {${\rm gr} M$ 
is an ideal of $\gr J_n R_f$}.

For $f \in \C[J_n \g]$ we denote by $\overline{f}$ its restriction to $J_n \Slo_f$. 
Then for $\xi=(\xi_i^{(j)} \mid i=1,\ldots,\ell, j= 0,\ldots,n) 
\in  J_n \C^\ell$, $(J_n  \psi_f )^{-1}(\xi)$ 
is the set of common zeroes of the ideal 
$$\mathscr{I}_{n,\xi} := \left( \overline{T^j p_i} - \xi_i^{(j)} \mid i=1,\ldots,\ell, j 
=0,\ldots,n\right).$$
{Here, note that $J_n \psi_f$ is the morphism: 
$$J_n \psi_f \colon \J \Slo_f \to J_n \C^\ell, \quad 
x \mapsto (T^j p_i  (x),i=1,\ldots,\ell,  j= 0,\ldots,n).$$}


%
%

\begin{Lem} \label{Lem:reduced}
Let $n \in\Z_{\ge 0} \cup\{\infty\}$  
and $\xi \in J_n \C^\ell$. 
Then the fiber $(J_n \psi_f )^{-1}(\xi)$ is reduced and irreducible. 
\end{Lem}

\begin{proof}
Clearly ${\rm gr} \mathscr{I}_{n,\xi} =\mathscr{I}_{n,0}$.

Let $a \in J_n R_f/\mathscr{I}_{n,\xi}$ and suppose that $a^k=0$ 
for some $k \in \Z_{\ge 0}$.
Then $\sigma_n(a)^k=0$, where $\sigma_n$ 
is the symbol of $a$ in ${\rm gr}(J_n R_f / \mathscr{I}_{n,\xi}) 
=J_n R_f / \mathscr{I}_{n,0}$.
As $\mathscr{I}_{n,0}$ is radical, 
$\sigma_n(a)=0$, and hence $a=0$. 
This proves that $\mathscr{I}_{n,\xi}$ is radical. 

Similarly, $J_n R_f/\mathscr{I}_{n,\xi}$
is a domain since ${\rm gr} (J_n R_f/\mathscr{I}_{n,\xi})
=J_n R_f/\mathscr{I}_{n,0}$ is.
Hence $\mathscr{I}_{n,\xi}$ is prime.
\end{proof}

{The Poisson bracket on $\g$ is algebraic 
since all adjoint orbits are open in their closure. 
From the inclusion $\overline{G.x \cap \Slo_f} \subset \overline{G.x}\cap \Slo_f$, 
 for $x \in \Slo_f$, 
we deduce that the symplectic leaf $G.x \cap \Slo_f$ of $\Slo_f$ 
is locally closed, and hence 
the Poisson bracket on $\Slo_f$ is algebraic. 
Indeed, $\overline{G.x}\cap \Slo_f$ is a 
finite union of symplectic leaves of $\Slo_f$ and so \cite[Proposition 3.7]{Brown-Gordon} 
applies. 
As a result, the hypothesis of Section \ref{sec:chiral-symplectic-leaves}
are satisfied. 

%

\begin{Pro} \label{Pro:fiber-cores}
For any $n \in \Z_{\ge 0}$, the fiber $(J_n \psi_f)^{-1}(0)$ 
is the closure of some $n$-chiral Poisson core in $J_n {\Slo_f}$, 
and $(\J \psi_f)^{-1}(0)$ 
is the closure of some chiral Poisson core in $\J  {\Slo_f}$. 
\end{Pro}
\begin{proof} 
{By Proposition \ref{Pro:zero-fiber}, }
$J_n (\psi_f^{-1}(0) ) \cong (J_n \psi_f)^{-1}(0)$ is irreducible for any $n \in \Z_{\ge 0}$. 
The statement follows from Corollary \ref{Co:equality-cores} 
because $\psi_f^{-1}(0)$ is the closure of some 
symplectic leaf. 
\end{proof}

\begin{Th} \label{Th:constant-fiber}
Let $z$ be in the vertex Poisson center of $\J R_f$, and 
$\xi \in \J (\g/\!/ G)$. Then $z$ is constant on $(\J \psi_f)^{-1}(\xi)$.
\end{Th}

\begin{proof}
By Proposition \ref{Pro:Poisson_center} and Proposition \ref{Pro:fiber-cores},  
any element $z$ in the vertex Poisson center of $\J R_f$ is constant 
on $(\J \psi_f)^{-1}(0)$. 
Let now $\xi \in \J (\g/\!/ G)$. 
Then the symbol $\sigma_\infty(z) \in {\rm gr}  
(\J R_f/\mathscr{I}_{\infty,\xi})$ 
belongs to the center $\mc{Z}(\J R_f/\mathscr{I}_{\infty,0})$ 
of ${\rm gr} (\J R_f/\mathscr{I}_{\infty,0})$. 
However, $\mc{Z}(\J R_f/\mathscr{I}_{\infty,0})\cong \C$ 
by the $\xi=0$ case. 
Therefore
 $\sigma_\infty(z)$ is constant,
and this happens only if $z$ itself is constant in $ \J R_f/\mathscr{I}_{\infty,\xi}$, 
that is, $z$ is constant on $(\J \psi_f)^{-1}(\xi)$. 
\end{proof}

\section{Vertex Poisson center and arc space of Slodowy slices} 
\label{sec:main-result}


The vertex Poisson algebra 
structure on $\C[\J R_f]$ can be described 
using cohomology of some dg-vertex Poisson algebras, 
which is a tensor product of functions over $\J\g$ 
with {\em fermionic-ghost} vertex Poisson super-algebra 
$\wedge^{\frac{\infty}{2}}(\mf{m})$, 
where $\m$ is a certain nilpotent algebra $\mf{m}$ 
(\cite[Theorem 4.6]{Ara09b}): 
$$\C[\J R_f]  \cong H^0(\C[\J \g] \otimes 
\wedge^{\frac{\infty}{2}}(\mf{m}),Q_{(0)}).$$
The canonical embedding 
$$\C[\J \g]   
\longrightarrow \C[\J \g] \otimes 
\wedge^{\frac{\infty}{2}}(\mf{m}), 
\quad f \longmapsto {f \otimes 1},$$ 
induces morphisms of vertex Poisson algebras, 
$$\mc{Z}(\C[\J \g] ) 
\longrightarrow \mc{Z}(\C[\J \g] \otimes 
\wedge^{\frac{\infty}{2}}(\mf{m})) 
\longrightarrow 
\mc{Z}(H^0(\C[\J \g] \otimes 
\wedge^{\frac{\infty}{2}}(\mf{m}),Q_{(0)})).$$
Hence we get a morphism of 
vertex Poisson algebras, 
$$\mc{Z}(\C[\J \g] ) 
\longrightarrow \mc{Z}(\C[\J R_f]) 
\subset  \C[\J R_f].$$ 
Note that this morphism corresponds to the restriction map. 

On the other hand, we have an isomorphism 
 \cite{RaiTau92,BeiDri96,EisFre01}: 
$\J(\g/\! / G) \cong \J \g / \!/ \J G$, where 
$\J \g / \!/ \J G= \Spec \C[\J \g]^{\J G}$. 
{In other words, the infinite jet scheme   
$\J G$ acts on $\J \g$ and 
$\C[\J \g]^{\J G}$ is the polynomial 
ring $ 
\C[\J (\g/\! / G)] = \C[T^j p_i, \, i=1,\ldots,\ell, j\ge 0]$. 
}
Therefore, we get the 
following isomorphisms: 
$$\mc{Z}(\C[\J \g] ) = \C[ \J \g / \!/ \J G] \cong \C[\J(\g/\! / G)].$$ 
So the above morphism 
from  $\mc{Z}(\C[\J \g] )$ to $\C[\J R_f]$ 
is nothing but the comorphism,  
$$(\J \psi_f)^* \colon 
\J(\g/\! / G) \to \J R_f,$$ 
of $\J \psi_f$. 

The map 
$(\J \psi_f)^*$ is an embedding.
Indeed,
let $\g_{reg}$ be the set of regular elements of $\g$, 
that is, those elements whose centralizer has minimal dimension $\ell$. 
Since the restriction of the morphism $\psi_f$ 
to $\Slo_f \cap \g_{reg}$ is smooth and surjective 
(see \cite{Kos78}, \cite[Section 5]{Pre02}), the restriction 
of 
$J_ n \psi_f$ 
to $ J_n (\Slo_f\cap  \g_{reg}) $ 
is smooth and surjective for any $n$ as well
(\cite[Remark 2.10]{EinMus} or \cite[\S3.4.3]{Fre07}). 
Therefore the morphism 
$(J_n \psi_f )^* \colon  \C[J_n \g]^{J_n G} \to J_n R_f$ 
is an embedding for any $n$. 
Moreover, the restriction 
of $\J \psi_f$ 
to $\J (\Slo_f\cap  \g_{reg}) $ 
is (formally) smooth and surjective, 
whence
the morphism 
$(\J \psi_f )^* \colon \C[\J \g]^{\J G} \to \J R_f$ 
is an embedding of vertex Poisson algebras. 

The aim of this section is to proof 
the following result. 

\begin{Th} \label{Th:main-result} 
The morphism $(\J \psi_f )^*$ induces an isomorphism 
of vertex Poisson algebras 
between $\C[\J \g]^{\J G}$ and the vertex Poisson center of 
$\C[\J \Slo_f]$.  
Moreover, $\C[\J \Slo_f]$ is free over its vertex Poisson center.  
\end{Th}

Note that 
$\C[\J \g]^{\J G}\cap \C[\g]=\C[\g]^G$
and
$\mc{Z}(\C[\J \Slo_f])\cap \C[\Slo_f]=\mc{Z}(\C[\Slo_f])$,
the Poisson center
of $\C[\Slo_f]$.
Hence
from the above theorem we recover the well-known 
result of Ginzburg-Premet  \cite[Question~5.1]{Pre07}
which states that
\begin{align}
\mc{Z}(\C[\Slo_f])\cong \C[\g]^G.
\label{Ginzburg-Premet}
\end{align}

To prove Theorem \ref{Th:main-result} we first state some preliminary results. 

\begin{Lem}[{\cite[Theorem A.2.9]{GooWal09}}]
\label{Lem:Goodman-Wallach}
Let $X,Y,Z$ be irreducible affine varieties. 
Assume that $f \colon X \to Y$ and $h \colon X \to Z$ are dominant 
morphisms such that 
$h$ is constant on the fibers of $f$. Then there exists a rational map 
$g \colon Y \to Z$ making the following diagram commutative: 
\xymatrix{ X \ar[r]^{f} \ar[d]_{h} & Y \ar@{-->}[ld]^{g} \\ Z &}
\end{Lem}

\begin{Lem}\label{Lem:rational-regular-morphisms}
Let $X$ and $Y$ be two 
normal irreducible affine varieties, 
and $f \colon X \to Y$ a flat morphism. 
Then $\C(Y) \cap \C[X] = \C[Y]$. 
Here, we view $\C[Y]$ as a subalgebra of $\C[X]$ 
using $f^* \colon \C[Y] \to \C[X]$.  
\end{Lem}
 
 \begin{proof}
 Since $X$ is normal and the fibers of $f$ 
 are all of dimension $\dim X - \dim Y$, 
 the image of the set $X'$ of smooth points of $X$ is an open subset 
 $Y'$ of $Y$ 
 such that $Y \setminus Y'$ has codimension at least $2$. 
 
 Let $y$ in $Y'$ and $x\in f^{-1}(y)\subset X'$. 
 Then we have a flat  extension of the local rings  
$\mc{O}_{Y,y} \to \mc{O} _{X,x}$. Since $\mc{O}_{Y,y}$ and 
$\mc{O} _{X,x}$ are regular local rings, 
they are factorial. For $a\in \C(Y)\cap \C[X]$, write 
$a = p/q$ with $p,q$ relatively prime 
elements of $\C[Y]$ .
Since $p,q$ are relatively prime, the multiplication by $p$ induces an injective 
homomorphism 
$$\mc{O}_{Y,y}/ q \mc{O}_{Y,y} \rightarrow  \mc{O}_{Y,y}/q  \mc{O}_{Y,y}.$$
Since $\mc{O} _{X,x}$ is flat over $\mc{O}_{Y,y}$, the base change 
$\mc{O} _{X,x}\otimes_{\mc{O}_{Y,y}} -$ yields an injective homomorphism 
$$\mc{O} _{X,x}/ q \mc{O} _{X,x}\rightarrow  \mc{O} _{X,x} /q \mc{O} _{X,x}.$$ 
Hence $p$ and $q$ are relatively prime in $\mc{O} _{X,x}$. 
In addition, the image of 1 is $0$ because $a=p/q$ is regular in $X$. 
As a result, $q$ is invertible in $\mc{O}_{X,x}$. 

Since the maximal ideal of $\mc{O}_{Y,y}$ 
is the intersection of  $\mc{O}_{Y,y}$ with the maximal ideal of  $\mc{O}_{X,x}$, 
$q$ is invertible in $\mc{O}_{Y,y}$, 
and so $a$ is in $\mc{O}_{Y,y}$. 
As a result, $a$ is regular on $Y'$ and then extends to a regular function 
on $Y$ since $Y$ is normal. 
 \end{proof}
 


\begin{proof}[Proof of Theorem \ref{Th:main-result}]
Let us prove the first assertion of the theorem. 
We view the algebra $J_n (\C[\g]^G)$ as a subalgebra of 
$J_n R_f$ for any $n$. 
We have already noticed that the inclusion  
$J_\infty (\C[\g]^G) \subset \mc{Z}(\J R_f)$ holds, with 
$\mc{Z}(\J R_f)$ the vertex Poisson center of $\J R_f$. 
Conversely, we have to prove that any element $z$ in the vertex 
Poisson center $\mc{Z}(\J R_f)$ 
can be lifted to a an element of $\J (\C[\g]^G)$. 

Since $\J \Slo_f$ is the projective limit of the projective system 
$(J_n \Slo_f, \pi_{m,n})$,  
the algebra $\J R_f$ is the inductive limit of the algebras $J_n R_f$. 
The injection $j_n$ from $J_n R_f$ to $\J R_f$ is defined by
$$j_n(\mu_n) (\gamma) := \mu_n(\pi_{\infty,n}(\gamma)),\qquad \mu_n \in J_n R_f, 
\; \gamma \in \J\Slo_f.$$

Let $z\in\mc{Z}(\J R_f) \subset \J R_f$. 
As $z \in \J R_f$, $z=z_n \in J_n R_f$ for $n$ big enough,  
where $z_n$ is such that 
$$z_n (\gamma_n) = z(j_n (\gamma_n)), \quad 
z(\gamma) = z_n(\pi_{\infty,n}(\gamma)), \qquad \gamma_n \in J_n \Slo_f, \; \gamma\in \J \Slo_f.$$ 
By Theorem \ref{Th:constant-fiber}, $z$ is constant on each fibers of $\J \psi_f$. 
As a consequence, $z_n$ is constant on each fibers of $J_n \psi_f$. 
Indeed, let $\xi_n \in J_n Y$, and  $\gamma_n,\gamma'_n \in J_n \psi_f^{-1}(\xi_n)$. 
Then $j_n(\gamma_n),j_n(\gamma'_n)$ are in $\J \psi_f ^{-1}(j_n(\xi_n))$ 
since 
$$\J \psi_f (j_n(\gamma_n)) = j_n(J_n \psi_f (\gamma_n))= j_n(\xi_n)= 
 j_n(J_n \psi_f (\gamma'_n))=\J \psi_f (j_n(\gamma'_n)) .$$
Hence,
\begin{align*}
z_n(\gamma_n) = z(j_n(\gamma_n)) =z(j_n(\gamma'_n)) =z_n(\gamma'_n)
\end{align*}
since $z$ is constant on $\J \psi_f ^{-1}(j_n(\xi_n))$. 

If $z$ is a constant function, that is, $z \in\C$, then clearly $z$ lies in 
the vertex Poisson center of $\J \C[\g^*]$. 
Hence, one can assume that $z$ is not constant. 
Furthermore, there is no loss of generality in assuming 
that $z$ is homogeneous for the 
{Slodowy grading} on $\J R_f$ 
because $\mc{Z}(\J R_f)$ 
is Slodowy invariant. Thus for any $t\in\C^*$, $t.z=t^k z$  
for some $k \in\Z_{\ge 0}$. So one can assume that 
the morphisms $z \colon \J \Slo_f \to \C$ 
and $z_n \colon J_n \Slo_f \to \C$ 
are dominant. 

Hence by Lemma \ref{Lem:Goodman-Wallach}, 
$z_n \in \C[J_n \Slo_f]$ induces a rational morphism $\tilde{z}_n$ 
on $J_n (\g /\! / G)$ since $z_n$ is constant on the fibers of the dominant 
morphism $J_n \psi_f$. 

As $\Slo_f$ and $\g/ \! / G$ are affine spaces, $J_n \Slo_f$ and 
$J_n( \g /\! / G)$ are affine spaces for any $n$. In particular, 
$J_n \Slo_f$ and 
$J_n( \g /\! / G)$ are normal and irreducible for any $n$. 
Therefore Lemma \ref{Lem:rational-regular-morphisms} 
can be applied because the morphism $J_n \psi_f \colon 
J_n \Slo_f \to J_n( \g /\! / G)$ is flat for any $n$. 
So, 
$\tilde{z}_n \in \C[J_n (\g /\! / G)]$. 
This holds for any $n$ such that $z=z_n$. 
Since $z=z_n$ for $n$ big enough, we deduce that 
$z$ can be lifted to an element of 
$\C[\J (\g /\! / G)]=\C[\J \g]^{\J G}$, whence the first part of the theorem.

It remains to prove the freeness. 
Since $\Slo_f \cap \mc{N}$ enjoys the same 
geometrical properties as $\mc{N}$, that is, 
$\Slo_f \cap \mc{N}$ is reduced, irreducible and is 
a complete intersection {with rational singularities}, the arguments of 
\cite[Theorem A.4]{EisFre01} can be applied 
in order to get that $\C[\J \Slo_f]$ is free over its vertex Poisson 
center (see also \cite[Proposition 2.5 (ii)]{ChaMor16}). This concludes the proof of the theorem. 
\end{proof}

\section{Center of W-algebras} \label{sec:W-algebras}
Let $V^k(\g)$ be the universal affine vertex algebra associated 
with $\g$ at level 
$k$, and 
let $\W^k(\g,f)$ be the (affine) W-algebraassociated with 
$(\g,f)$ at level $k \in\C$. 
The W-algebra $\W^k(\g,f)$ is defined by the quantized Drinfeld-Sokolov 
reduction associated with $f$
(\cite{FF90,KacRoaWak03}).

The embedding 
$Z(V^k(\g)) \hookrightarrow V^k(\g)$ induces a vertex algebra 
homomorphism 
$$Z(V^k(\g)) \longrightarrow Z(\W^k(\g,f))$$ 
for any $k \in\C$. Here, for $V$ a vertex 
algebra $V$, $Z(V)$ denotes the vertex center of $V$, 
that is, 
$$Z(V)=\{z \in V \mid a_{(n)} z = 0 \text{ for all } a \in V, n \ge 0\}.$$
Both $Z(V^k(\g))$ and $Z(\W^k(\g,f))$ are trivial unless $k=cri$ 
is the critical level 
$cri=-h^\vee$ with $h^\vee$ the dual Coxeter number of $\g$. 
For $k=cri$,   $\mf{z}(\affg):=Z(V^{cri}(\g))$ is known as the 
{Feigin-Frenkel 
center} \cite{FeiFre92}. 

\begin{Th} \label{Th:W-algebra}
The embedding $\mf{z}(\affg)\hookrightarrow V^{cri}(\g)$ induces an isomorphism 
$$\mf{z}(\affg) \stackrel{\sim}{ \longrightarrow} Z(\W^{cri}(\g,f))$$
and we have 
$ \gr Z( \W^{cri}(\g,f)) \cong \mc{Z} (\C[\J \Slo_f])$.
\end{Th}

\begin{proof}
Recall that there is an obvious
vertex algebra homomorphism 
$\mf{z}(\affg) \to Z(\W^{cri}(\g,f))$, see \cite{A11}. 
Hence it is sufficient to show that the induced homomorphism
$\gr \mf{z}(\affg)\ra \gr  Z(\W^{cri}(\g,f))$ is an isomorphism.

First, 
we have (\cite{FeiFre92}) 
$$\gr \mf{z}(\affg) \cong \C[\J \g]^{\J G}.$$
On the other hand,
we have (\cite[Theorem 4.17]{Ara09b}) 
$$\gr \W^{cri}(\g,f) \cong \C[\J \Slo_f], $$ 
and so 
$$\mc{Z} (\gr \W^{cri}(\g,f) ) \cong \mc{Z} (\C[\J \Slo_f]).$$
By Theorem \ref{Th:main-result}, $\mc{Z} (\C[\J \Slo_f]) \cong  \C[\J \g]^{\J G}$, 
which forces the compound map 
$$\mc{Z} (\C[\J \Slo_f]) \cong 
\gr \mf{z}(\affg) \longrightarrow \gr Z( \W^{cri}(\g,f))  \hookrightarrow 
\mc{Z} (\gr \W^{cri}(\g,f) ) \cong \mc{Z} (\C[\J \Slo_f])$$ 
to be an isomorphism. 
This completes the proof.
\end{proof}
Theorem \ref{Th:W-algebra} was stated in \cite{A11}, but the proof of the surjectivity was incomplete.
 
Note that
the similar argument  as above using \eqref{Ginzburg-Premet}
recovers 
Premet's result \cite{Pre07} stating that the center of the {\em finite 
W-algebra} $U(\g,f)$ associated with $(\g,f)$ 
is isomorphic to the center of the enveloping algebra $U(\g)$ 
of $\g$. 

\newcommand{\etalchar}[1]{$^{#1}$}



\end{document}